\documentclass[11pt]{article}
\usepackage{epsfig}
\usepackage{amsmath,amssymb,amsthm,url,verbatim,pdfsync}
\usepackage{mathtools} 
\usepackage{mathrsfs}   
\usepackage[utf8]{inputenc}
\usepackage{hyperref}

\usepackage{tikz,eucal,enumerate,mathrsfs,bbm}
\usepackage{tikz-cd}

\numberwithin{equation}{section}
\numberwithin{figure}{section}

\newtheorem{thm}{Theorem}[section]
\newtheorem{theorem}[thm]{Theorem}
\newtheorem{lemma}[thm]{Lemma}
\newtheorem{prop}[thm]{Proposition}

\newtheorem{defn}[thm]{Definition}

\def\R{{\mathbb R}}
\def\vphi{{\varphi}}

\def\F{{\cal F}}
\def\G{{\cal G}}

\def\L{{\cal L}}
\def\T{{\cal T}}

\def\calP{{\cal P}}
\def\calE{{\cal E}}

\def\Bbar{{\overline{B}}}

\def\Abar{{\overline{A}}}
\def\Vbar{{\overline{V}}}
\def\Gbar{{\overline{G}}}
\def\vbar{{\overline{v}}}
\def\abar{{\overline{a}}}
\def\cbar{{\overline{c}}}
\def\fbar{{\overline{f}}}

\def\cleq{\preccurlyeq}
\def\cgeq{\succcurlyeq}

\def\la{\langle}
\def\ra{\rangle}
\def\pa{{\partial}}
\def\ep{\epsilon}
\def\nn{\nonumber}
\def\vspc{{\vspace{-0.2in}}}

\def\gmax{{g_{max}}}
\def\gmin{{g_{min}}}
\def\hbar{{ \bar{h}}}

\def\xdot{{\dot{x}}}

\def\gammadot{{\dot{\gamma}}}

\def\alphabar{{\bar{\alpha}}}

\def\ubar{{\bar{u}}}
\def\vbar{{\bar{v}}}

\def\wbar{{\bar{w}}}

\def\Tbar{{\bar{T}}}

\def\fhat{{\hat{f}}}

\def\tWF{{\text{WF}}}
\def\tsupp{{\text{supp}}}

\title{Fixed angle inverse scattering with non-constant velocity}
\author{ 
Lauri Oksanen\thanks{Department of Mathematics and Statistics, University of Helsinki, PO Box 68, 00014 Helsinki,
Finland. Email: lauri.oksanen@helsinki.fi},
~
Rakesh\thanks{Department of Mathematical Sciences, University of Delaware, Newark DE 19808, USA.
Email: rakesh@udel.edu},
~
Mikko Salo\thanks{Department of Mathematics and Statistics, P.O. Box 35 (MaD), FI-40014 University of
Jyväskylä, Finland. Email: mikko.j.salo@jyu.fi}
} 
\date{July 17, 2026}

\begin{document}
\maketitle
\tableofcontents

\begin{abstract}
In this article, we study formally determined inverse problems for wave equations in the presence of a variable 
sound speed. We prove that by measuring the boundary data of finitely many plane waves and their 
complementary solutions, one can uniquely recover the unknown coefficients of the highest order terms of a second 
order hyperbolic operator with time independent coefficients. This improves earlier 
rigidity results in \cite{ors26a}, \cite{ors26b} which compared the wave operator generated by a Riemannian metric
with the wave operator generated by the Euclidean metric. We compare two general second order hyperbolic 
operators with time independent coefficients, with the same lower order terms. However, we 
require the geometry associated with {\bf one of the } operators to satisfy a pseudoconvexity condition, 
a no-caustics condition, and a spanning condition. In particular one of the operators could be a wave operator with the 
sound speed close to a constant and the other operator could be arbitrary. 

To prove the results, we introduce the notion of a complementary solution for a generalized plane wave solution 
generated by an incoming plane wave. The complementary solution extends smoothly, across an interface, the generalized plane wave. The unknown coefficients appear in a transport equation at the interface. We show that the unknown coefficients in the interior can be extracted from this transport equation, from the boundary data, 
via a sequence of Carleman estimates for the wave operator.
\end{abstract}

%%%%%%%%%%
%%%%%%%%%%%

\section{Introduction}\label{sec:intro}
 
Our goal is to determine the acoustic properties of an inhomogeneous medium from the medium response, measured on 
the boundary of the inhomogeneous part,
to sources based in the homogeneous part. So, neither the sources nor the measurements are generated or measured 
in the unknown 
inhomogeneous part of the medium. There is substantial work on problems for media where waves travel 
with constant velocity so the unknown properties are associated with the lower order terms of the differential 
operator modeling the medium. We study problems for media where the waves may not travel with 
constant velocity and our goal is to recover the velocity (isotropic or anisotropic)
or the lower order terms of the differential operator. The problems for media with non-constant velocities present 
substantial challenges because of the complicated structure of the solutions of the forward problem, compared to problems 
for media with constant velocities. 
Further, our focus is on problems where the data comes from only a 
finite number (dimension dependent) of experiments - the formally determined problems.

Let $B$ be the origin centered open ball of radius $1$ in $\R^n$, $n \geq 2$. We say the differential operator
\[
\L := \pa_t^2 - \sum_{i,j=1}^n a_{ij}(x) \pa_i \pa_j - \sum_{i=1}^n b_i(x) \pa_i - c(x),
\]
on $\R^n \times \R$, is {\bf admissible}, if $a_{ij}(x), b_i(x), c_i(x)$ are smooth real valued functions on $\R^n$ with
$A(x) = (a_{ij}(x))$ positive definite and $A(x)-I, b_i(x), c(x)$ supported in the open ball $B$. We also
consider the geometric operator 
\[
\Box_g := \pa_t^2 - \Delta_g
\]
where $g$ is a smooth Riemannian metric on $\R^n$ and $\Delta_g$ is the Laplace-Beltrami operator associated with 
$g$. In coordinates, if $g= (g_{ij})$, $g^{-1} = (g^{ij})$ and $|g|$ denotes the determinant of $g$, then
\[
\Delta_g = |g|^{-1/2} \pa_i \left ( |g|^{1/2} g^{ij} \pa_j \right ) = g^{ij} \pa_i \pa_j + |g|^{-1/2} \pa_i \left ( |g|^{1/2} g^{ij} \right ) \pa_j .
\]
 A smooth Riemannian metric $g$ on $\R^n$ is said to be {\bf admissible} if $g- g_{Eucl}$ is supported in $B$; here
 $g_{Eucl}$ is the Euclidean metric on $\R^n$.
{\em Throughout the article, we work only with admissible $\L$ or with admissible metrics $g$}, and we use
the Einstein summation convention. We also define the operator
\[
\calE := a_{ij} \pa_i \pa_j + b_i \pa_i + c.
\]
For any matrix $M$, $\|M\|_\infty$ will denote the largest of the
absolute values of the entries of $M$. Throughout the article, we assume $n$ is a positive integer, $n \geq 2$.

We consider two `fixed angle' inverse scattering problems. 
\vspc
\begin{itemize}
\item (The $c$ determination problem). 
Fix a unit vector $\omega$ in $\R^n$ and suppose $A(x), b(x)$ are known.
Let $U_c$ be the solution of the initial value problem
\begin{subequations}
\begin{align}
\L U_c =0, & \qquad \text{on } \R^n \times \R,
\label{eq:Ucde}
\\
U_c(x,t) = \delta(t- x \cdot \omega), & \qquad \text{on } \R^n \times (-\infty, -1).
\label{eq:Ucic}
\end{align}
\end{subequations}
Is $\F : c \to U_c|_{\pa B \times (-\infty, T)}$ injective if $T$ is large enough?
\item (The $A$ determination problem) Suppose $b,c$ are known and $\Omega$ is a finite set of unit vectors in $\R^n$. For each $\omega \in \Omega$, let $U_{A,\omega}$ be the solution of the initial value problem
\begin{subequations}
\begin{align}
\L U_{A,\omega} =0, & \qquad \text{on } \R^n \times \R,
\label{eq:UAde}
\\
U_{A,\omega}(x,t) = H(t- x \cdot \omega), & \qquad \text{on } \R^n \times (-\infty, -1).
\label{eq:UAic}
\end{align}
\end{subequations}
Is $\F : A \to [U_{A,\omega}|_{\pa B \times (-\infty, T)}]_{\omega \in \Omega}$ injective if $T$ is 
large enough? 

Since the goal is to recover $A(x) = (a_{ij}(x))$ which is 
made up of $n(n+1)/2$ functions, one is likely to need data generated by
incoming wave sources from at least $n(n+1)/2$ different directions
$\omega$. Hence the need for data associated with a finite set of unit vectors in $\Omega$.
\end{itemize}

There is a $g$ determination problem, whose statement is similar to the 
$A$ determination problem, except the 
operator $\L$ is replaced by the operator $\Box_g$ and the goal is the recovery of $g$ instead of $A$. 
One cannot expect to recover $g$ from $\F(g)$ because, for any diffeomorphism 
$\Psi : \R^n \to \R^n$ with $\Psi(x) = x$ for $x$ in $\R^n \setminus B$, we have $\F(g) = \F(\psi^*(g))$. So one 
hopes to recover $g$ only up to diffeomorphisms of the type $\Psi$. The $g$ determination problem is studied in a 
forthcoming manuscript and its generalization to Lorentzian metrics (though with slightly different data) is studied 
in the companion article \cite{ors26d}.

For the $c$ determination problem, we use only one incoming direction $\omega$, so we do not display 
the dependence of $U_c$ on 
$\omega$. For the $A$ determination problem, we use several incoming directions $\omega$, so we display 
the dependence on $\omega$ of $U_{A,\omega}$. 
We use a delta plane wave source to define $U_c$ but use a Heaviside plane wave source to define
$U_{A,\omega}$. This is for convenience because the first term in the progressing wave expansion of $U_c$ is independent of $c$, while the the first term in the progressing wave expansion of 
$U_{A,\omega}$ does depend on $A$. One could have used an incoming Heaviside plane wave to 
define a new $U_c$ but the old $U_c$ would just be the $t$ derivative of the new $U_c$.

The $A$ determination problem is more challenging than the $c$ determination problem since a variation in 
$A$ (so a variation in the speed of propagation of waves) has a more complex impact on $U_{A,\omega}$ 
than a variation of the lower order coefficient $c$ has on $U_c$. Our techniques should carry over to the problem 
of determining $b$ for known $A,c$. For the $g$ determination problem, in addition to the invariance of 
$\F(g)$ under certain diffeomorphisms, there is the added complexity that the operator $\Box_g$ depends on 
the first order derivatives of $g$ in addition to its dependence on $g$, 
unlike the $A$ and $c$ determination problems. Hence the $g$ determination problem presents more 
challenges than the $A$ or $c$ determination problems. These additional challenges have been resolved in the forthcoming manuscript for the $g$ determination problem and its generalization (with slightly different data) to Lorentzian metrics in the companion article \cite{ors26d}.

In \cite{rs20a}, we proved a stability result for the $c$ determination problem but only when $A=I, b=0$ and 
we needed the data generated by two sources - a plane wave coming from direction $\omega$ and a plane 
wave coming from direction $-\omega$. In \cite{ms22}, this result for the $c$ determination problem was 
generalized to the $c$ determination problem for the operator $\Box_g + c$, where $g$ need not be the Euclidean metric but $g$ had to satisfy
three geometrical conditions - a `no caustics' condition for the geodesics originating in region $x \cdot \omega \leq -1$ with initial velocity $\omega$, a geometrical symmetry condition for these geodesics, and the existence
of a function on $\Bbar$ which is convex w.r.t the geodesics of $g$.

{\em For the $c$ determination problem, we obtain a uniqueness result where $g=A^{-1}$ is not required to 
satisfy the geometrical symmetry condition in \cite{ms22}, but must still satisfy the other two conditions 
mentioned in the previous paragraph.} However, for our result, 
the second piece of data is not the medium response to a plane wave source coming from the direction 
$-\omega$. Instead, the second piece of data is the medium response to a boundary source constructed from 
the data 
generated by the first source $\delta(t-x \cdot \omega)$  - without 
knowing $A,b,c$ inside $B$. Theorem \ref{thm:cunique}, stated in section \ref{sec:main}, is a 
careful statement of the result.

In \cite{ors26a} we studied the $g$ determination problem and showed that $\F(g) = \F(g_{Eucl})$ 
implies $g=g_{Eucl}$ up to a diffeomorphism of $\R^n$ which is the identity outside $B$; we needed a large enough $T$ and we used
\[
\Omega = \{ e_i : i=1, \cdots, n \} \cup \{ (e_i + e_j)/\sqrt{2} : i,j=1, \cdots, n, ~ i \neq j\}.
\]
A simple adaptation of the proof in \cite{ors26a} shows that for admissible $\L$ with $b=0$, $c=0$, we have 
$\F(A) = \F(I)$ if and only if  $A=I$, for the above $\Omega$ and $T$ large enough;
note the lack of a gauge invariance. The same proof also shows that if $A = \rho^{-1} I$ for some positive smooth 
function $\rho$ on $\R^n$ with 
$\rho=1$ outside $B$ and $b=0, ~c=0$, then $\F (A) = \F(I)$ if and only if $\rho=1$, for $T$ large enough
and $\Omega$ having just a single
element. In \cite{ors26b}, we generalized the result in \cite{ors26a} to the operator
$\Box_h$ where $h(x,t)$ is a Lorentzian metric on $\R^n \times \R$ with a temporal function, $h$ the Minkowski 
metric outside $B \times \R$, $h$ independent of $t$ for large $t$, and the comparison was with the Minkowski metric. 

{\em For the $A$ determination problem, we obtain a uniqueness result  
where the comparison is not just with $A_*=I$ but with a more general class of $A_*$.} However, for our result, we  
have restrictions on the $A_*$ being compared to and the second piece of data is similar to the one mentioned 
above in the description of the result for the $c$ determination problem. 
For our result, the metric $g=A_*^{-1}$ must satisfy the `no 
caustics' condition, and the existence of a function on $\Bbar$ which is convex w.r.t the geodesics of $g$.
We also need a `spanning condition' on the set of vectors $\{\nabla \alpha_\omega\}_{\omega \in \Omega}$, 
where $\alpha_\omega$ are certain solutions of the eikonal equation for $g_*=A_*^{-1}$. 
Theorem \ref{thm:Aunique}, stated in section \ref{sec:main}, is a careful statement of the result.

Here is a brief list of some novel technical points of this article.
\vspc
\begin{itemize}
\item 
We prove uniqueness results for unknown sound speeds with one of the two sound speeds being 
compared required to satisfy certain properties, but neither of these sound speeds needs to be constant. 
This generalizes the results in \cite{ors26a} \cite{ors26b} where one of 
the two sound speeds being compared was required to be constant.
\item 
To each distorted plane wave solution we associate a complementary solution, so that the sum of  
the distorted plane wave solution and its complementary solution is smooth enough across the characteristic 
surface where the distorted plane wave is singular. For our uniqueness result, we need
the boundary data from the distorted plane wave solution as well as for its complementary solution.
\item 
The one sided trace of the distorted plane wave solution on the characteristic surface (where the solution is 
singular) satisfies a transport equation which relates the trace to a solution of the eikonal equation associated
with the unknown sound speed. This relation is critical to prove the uniqueness result using
a modified Bukhgeim-Klibanov method (see \cite{rs20a} and \cite{ms22}) which uses Carleman estimates for inverse problems in a clever way. For the sound speed problem, one must use
an additional trick, due to Romanov \cite{romanov2002}, for adjusting the domains of the distorted plane
wave solutions.
\item 
For matrix valued sound speeds, we introduce a matrix spanning condition to relate the unknown sound speeds to eikonal solutions.
\end{itemize}

%%%%%%%%%%%%%%%
%%%%%%%%%%%%%%%

\section{Statements of the main results}\label{sec:main}

IVP stands for Initial Value Problem, IBVP stands for Initial Boundary Value Problem, and
\[
P(x) \cleq Q(x) \qquad \text{or} \qquad P(x) \cgeq Q(x) \qquad \qquad \text{for all } x \in K
\]
means there is a constant $C>0$, independent of $x \in K$, such that
\[
P(x) C \leq C \, Q(x) \qquad \text{or} \qquad P(x)  \geq C \, Q(x) \qquad \qquad \text{for all } x \in K.
\]

The principal symbol of $\L$ is $-\tau^2 + a_{ij}(x) \xi_i \xi_j$ with $(x,t;\xi,\tau) \in T^*(\R^n \times \R)$,
which is the same as the principal symbol of $\Box_g$ with $g=A^{-1}$. As seen in \cite{ors26a}, the null bicharacteristics 
of $\Box_g$ are associated with the unit speed geodesics of $(\R^n,g)$, so the null bicharacteristics of $\L$ are associated 
with the unit speed geodesics of the Riemannian metric $g=A^{-1}$. Throughout the article, the Riemannian metric 
associated with $\L$ will be $g=A^{-1}$. Note that if $\L$ is admissible then $g=A^{-1}$ is an admissible Riemannian 
metric.

For an admissible Riemannian metric $g$ on $\R^n$ define
\[
\la v,w \ra := g(x)(v,w), ~~ \|v\| := \sqrt{ \la v, v \ra }, \qquad v,w \in T_x(\R^n).
\]
We reserve $|v|$ and $v \cdot w$ for the Euclidean norm and dot product. There are optimal positive constants
$g_{min}, g_{max}$ such that
\[
g_{min} |v|^2 \leq g(x) (v,v) \leq g_{max} |v|^2, \qquad  \forall v \in T_x(\R^n), ~ x \in \R^n
\]
where $|v|^2$ is its Euclidean norm. The unit speed geodesics of $g$ travel with Euclidean speeds between
$1/\sqrt{\gmax}$ and $1/\sqrt{\gmin}$. So for $g=A^{-1}$, $g_{min} = A_{max}^{-1}$ and $g_{max}= A_{min}^{-1}$
where $A_{min}, A_{max}$ are the optimal positive constants such that 
\[
A_{min} |v|^2 \leq v^T A(x) v  \leq A_{max} |v|^2, \qquad \forall v \in T_x(\R^n), ~ \forall x \in \R^n,
\]
and the geodesics of $g=A^{-1}$ travel with Euclidean speeds between $\sqrt{A_{min}}$ and $\sqrt{A_{max}}$.

For the admissible Riemannian metric $(\R^n,g)$, a covector $\xi \in T_x^*(\R^n)$ may be identified with a vector 
$v \in T_x(\R^n)$ through
\[
\xi(w) = g(x)(v,w), \qquad \forall w \in T_x(\R^n).
\]
In coordinates, $\xi = (g_{ij}(x))v$ so $v = g(x)^{-1}\xi$. 
For any $f \in C^\infty(\R^n)$, define $(\nabla_g f)(x) \in T_x(\R^n)$ to be the vector
identified with the covector $(\nabla f)(x) \in T_x^*(\R^n)$; in coordinates 
\[
(\nabla_g f)(x) = (g_{ij}(x))^{-1} (\nabla f)(x).
\]
If $\gamma: [c,d] \to \R^n$ is a (piecewise-smooth continuous) curve on $\R^n$, its length is defined to be
\[
L(\gamma):= \int_c^d \| \gammadot(r) \| \, dr.
\]
For points $p,q$ in $\R^n$, we define the distance between $p$ and $q$ as
\[
d_g(p,q) = \inf \{ L(\gamma): \text{ $\gamma$ is a curve from $p$ to $q$} \}.
\] 
One knows that the infimum is attained and attained by a geodesic and $(\R^n,d_g)$ is a complete metric space - see 
\cite{ors26a}.

Suppose $(\R^n, g)$ is an admissible Riemannian metric.
Fix a unit vector $\omega$ in $\R^n$ and define the hyperplane
\[
P_\omega = \{ p \in \R^n : p \cdot \omega = -1 \}.
\]
For $p \in P_\omega$, let $\gamma_{p,g,\omega} : \R \to \R^n$ denote the geodesic of $(\R^n,g)$ with 
\[
\gamma_{p,g,\omega}(-1)=p, \qquad  \gammadot_{p,g,\omega}(-1) = \omega.
\]
One can show that $\gamma_{p,g,\omega}$ has constant speed (in the Riemannian metric) and, due to the bounded Euclidean speed of propagation, 
$\gamma_{p,g,\omega}(r)$ is defined for all $r \in \R$.
Define the {\em time of first arrival} function $\alpha_{g,\omega}: \R^n \to \R$ with
\[
\alpha_{g,\omega}(x) := x \cdot \omega,  \qquad \text{for $x \in \R^n$ with $x \cdot \omega \leq -1$},
\]
and
\[
\alpha_{g,\omega}(x) := \inf \{ d_g(x,p) -1 : p \in P_\omega \}, \qquad \text{for $x \in \R^n$ with $x \cdot \omega \geq -1$}.
\]
We have shown in \cite{ors26a} that this infimum is attained and, for each $x \in \R^n$,  there is a 
$p \in P_\omega$ and an $r \in \R$ such that
\[
\gamma_{p,g,\omega}(r) = x, ~~ r = \alpha_{g,\omega}(x).
\]
Further $\alpha_{g,\omega}$ is Lipschitz continuous on $\R^n$.

For a fixed unit vector $\omega$ in $\R^n$ and an admissible metric $(\R^n,g)$ we construct a Lagrangian 
submanifold of $T^*(\R^n \times \R)$, which turns out to be the wave front set of 
of $U_c$ and of $U_{A,\omega}$, if we take $g=A^{-1}$.
For each $p \in P_\omega, \tau \in \R$, let $(x(t,p,\tau), \xi(t,p,\tau))$ 
be\footnote{The $x(t,p,\tau), \xi(t,p,\tau)$ depend on
$\omega$ but, to avoid cumbersome notation, we do not show the dependence on $\omega$.}
the solution of the IVP
\begin{subequations}
\begin{align}
\xdot = - \frac{1}{\tau} g^{-1} \xi, & \qquad \dot{\xi}_k = - \frac{1}{2\tau} (g^{-1} \xi)^T \pa_{x_k}(g) (g^{-1} \xi),
 \qquad k=1, \cdots, n,
 \label{eq:nullde}
\\
x(t=-1,p,\tau) = p, & \qquad \xi(t=-1, p, \tau) = - \tau \omega.
\label{eq:nullic}
\end{align}
\end{subequations}
In \cite[Section 2.2.1]{ors26a}, it was shown that this IVP has a unique solution, and this 
solution exists for all $t \in \R$. Further, $x(t,p,\tau)$ is independent of $\tau$, so we write $x(t,p)$ instead of $x(t,p,\tau)$. 
It was also shown that
\begin{equation}
x(t,p) = \gamma_{p,g,\omega}(t), \qquad \xi(t,p,\tau) = - \tau g(x(t,p)) \gammadot_{p,g,\omega}(t).
\label{eq:covecvec}
\end{equation}
Define
\begin{equation}
\Lambda_{g,\omega} := \{ \left ( x(t,p), t; \xi(t,p,\tau), \tau \right ) : p \in P_\omega, ~ t \in \R, ~ \tau \in \R, ~ \tau \neq 0 \};
\label{eq:lambdadef}
\end{equation}
from standard theory, $\Lambda_{g,\omega}$ is a Lagrangian submanifold of $T^*(\R^n \times \R)$. 
Noting \eqref{eq:covecvec}, we identify $\Lambda_{g,\omega}$ with
\begin{equation}
\Lambda_{g,\omega} := \{ \left (\gamma_{p,g,\omega}(t),t; \gammadot_{p,g,\omega}(t), 1 \right ) : 
p \in P_\omega \},
\end{equation}
a subset of the unit sphere bundle (actually the $\sqrt{2}$ length bundle) of $T(\R^n \times \R)$.  
If $K$ is a subset of $\R^n$ or $\R^n \times \R$, we define
\[
\Lambda_{g,\omega}|_K = \Lambda_{g,\omega} \cap \pi^{-1}(K),
\]
where $\pi$ is the relevant projection of $T^*(\R^n \times \R)$ onto $\R^n$ or $\R^n \times \R$. 

If $\Omega$ is an open subset of $\R^m$ then $H^s(\Omega)$ denotes the usual Sobolev space and we say 
$u \in H^s_{loc}(\Omega)$ if $u \in H^s(D)$ for every bounded open subset $D$ of $\Omega$.

The following proposition claims the existence and uniqueness of the solutions $U_c, U_{A,\omega}$. Its proof is placed 
in the appendix as its proof is similar to the proofs in \cite{ors26b} for the forward problem.
\begin{prop}[The forward problem]\label{prop:uforward}
Suppose $\omega$ is a unit vector in $\R^n$, $\L$ is an admissible operator on $\R^n \times \R$, and $g=A^{-1}$. 
\begin{enumerate}[(a)]
\item The IVP \eqref{eq:Ucde}, \eqref{eq:Ucic} has a unique distributional solution $U_c$. Further, 
$U_c \in H^{-1}_{loc} (\R^n \times \R)$, $\tWF(U_c) = \Lambda_{g,\omega}$, and 
$\tsupp(U_c) \subset \{ t \geq \alpha_{g,\omega}(x) \}$.
\item The IVP \eqref{eq:UAde}, \eqref{eq:UAic} has a unique distributional solution $U_{A,\omega}$. Further, 
$U_{A,\omega}\in L^2_{loc} (\R^n \times \R)$, $\tWF(U_{A,\omega}) = \Lambda_{g,\omega}$, and 
$\tsupp(U_{A,\omega}) \subset \{ t \geq \alpha_{g,\omega}(x) \}$.
\end{enumerate}
\end{prop}
\vspc
Since $\Lambda_{g,\omega}$ does not intersect the normal bundle of $\pa B \times \R$, 
\cite[Corollary 8.2.7]{hor85} implies that
$U_c, U_{A,\omega}$ have traces on $\pa B \times \R$,  hence the maps $\F$, used to state the
$c,A$ recovery problems, are well defined. 

For our result for the $c$ problem and the $A$ problem, the metric $g=A^{-1}$ must satisfy a `no caustics' condition with 
respect to its geodesics originating in the region $x \cdot \omega \leq -1$ with initial velocity $\omega$. We state this condition carefully next.
%
%%%%%%%%%%%%%%%%%%%%%%%
\begin{defn}\label{def:Phi}
Suppose $(\R^n,g)$ is an admissible Riemannian metric, $\omega$ is a unit vector in $\R^n$ and $T_0>1$. 
We say the collection $(g,\omega,T_0)$ satisfies the {\em Exterior Injectivity Condition (EIC)} if the map 
$\Phi_{g,\omega} : P_\omega \times (-\infty, T_0) \to \R^n$ with 
\[
\Phi_{g,\omega}(p,r) = \gamma_{p,g,\omega}(r), \qquad p \in P_\omega, ~ r \in (-\infty, T_0),
\]
is injective on $\Phi_{g,\omega}^{-1}( \R^n \setminus B)$, and there is a $T_1<T_0$ such that
\[
\{ x \in \R^n : x \cdot \omega =1 \} \subset \Phi_{g,\omega}(P_\omega \times (-\infty, T_1]).
\]
\end{defn}
\vspc
Define $D_{g,\omega}$ to be range of $\Phi_{g,\omega}$, that is
\[
 D_{g,\omega} := \Phi_{g,\omega}( P_\omega \times (-\infty, T_0) ).
\]
The following proposition, needed in our proofs and an immediate 
consequence\footnote{
The hypothesis of  \cite[Proposition 2.3]{ors26a} requires $T_0 > 2 \sqrt{g_{max}} -1$ but that proof goes through if
we are just guaranteed the existence of the $T_1$, and that $T_1<T_0$. 
}
of \cite[Proposition 2.3]{ors26a}, asserts
that if $(g,\omega,T_0)$ satisfies the EIC then $\Phi_{g,\omega}$ is a diffeomorphism. 
\begin{prop}[EIC implies diffeomorphism]\label{prop:exinj}
Suppose $(\R^n,g)$ is admissible, $\omega$ is a unit vector in $\R^n$, and $T_0>1$. 
If $(g,\omega,T_0)$ satisfies EIC then $\Phi_{g,\omega}$ is a diffeomorphism onto its range $D_{g,\omega}$, 
$D_{g,\omega}$ is an open subset of $\R^n$, and
\[
\{ x \in \R^n : x \cdot \omega \leq 1 \} \subset \Phi_{g,\omega}( P_\omega \times (-\infty, T_1] ),
\]
where $T_1$ is given in the definition of the EIC property.
Further, $D_{g,\omega} = \{ x \in \R^n : \alpha_{g,\omega}(x) < T_0 \}$ and $\alpha_{g,\omega}$ is a smooth function on $D_{g,\omega}$ with the following properties:
\vspc
\begin{itemize}
\item if $x \in D_{g,\omega}$ then $x = \gamma_{p,g,\omega}(r)$ iff $r = \alpha_{g,\omega}(x)$;
\item $\|\nabla_{g,\omega} \alpha_{g,\omega}\|^2=1$ on $D_{g,\omega}$;
\item for $x \in D_{g,\omega}$, $(\nabla_g \alpha_{g,\omega})(x) = \gammadot_{p,g,\omega}(r)$ for the unique
$p \in P_\omega$, $r<T_0$ with $\gamma_{p,g,\omega}(r) = x$.
\end{itemize}
\end{prop}

\vspc
Our results for the $c,A$ problem require the hypothesis 
that $(g=A^{-1},\omega,T_0)$ satisfies EIC. We have two observations regarding the verification of the EIC.
\begin{itemize}
\item The injectivity condition in the EIC is equivalent to the statement that, for a fixed 
$\omega,g$, and $p$ varying in $P_\omega$,
the geodesics $r \to \gamma_{p,g,\omega}(r)$ do not intersect or self intersect outside $B$ before time $T_0$.
\item For a fixed $\omega$ and a fixed eligible $\L$ with unknown $A$ or $c$, 
the EIC condition for $(g=A^{-1}, \omega, T_0)$ 
can be verified from the inverse problem data $U_{A,\omega}|_{\pa B \times (-\infty, T_0)}$ or 
$U_{c}|_{\pa B \times (-\infty, T_0)}$ because, for example, knowing $U_{A,\omega}|_{\pa B \times (-\infty, T_0)}$, 
one can determine $U_{A,\omega}$ on $(\R^n \setminus B) \times (-\infty, T_0)$ by solving 
an exterior IBVP. Then, using the tangent bundle interpretation of 
$\Lambda_{g,\omega}$ and that $\tWF(U_{A,\omega}) = \Lambda_{g,\omega}$ (for $g=A^{-1}$), one can check the 
non-intersection of geodesics condition and that the hyperplane $x \cdot \omega =1$ is a subset of the projection 
of the singular support of $U_{g,\omega}|_{ t \leq T_1}$ for some $T_1 < T_0$. 
\end{itemize}
\vspc
The following proposition, about the transfer of the EIC property from one operator (metric associated with the operator) to 
another, will be useful for us.  It is a quick consequence of Proposition \ref{prop:exinj} and the 
proof is given in the appendix.
\begin{prop}[EIC transfer]\label{prop:eictransfer}
Let $\omega$ be a unit vector in $\R^n$ and $T_0>1$.
Suppose $\L := \pa_t^2 - a_{ij} \pa_i \pa_j - b_i \pa_i - c$ and $\L' := \pa_t^2 - a'_{ij} \pa_i \pa_j - b_i \pa_i - c $ are 
admissible and $(g=A^{-1},\omega, T_0)$ satisfies EIC. If
$U_{A,\omega} = U_{A',\omega}$ on $\pa B \times (-\infty, T_0)$ then $(g'=A'^{-1},\omega,T_0)$ satisfies
EIC, $D_{g,\omega} = D_{g',\omega}$, and $\alpha_{g,\omega} = \alpha_{g',\omega}$ on $\R^n \setminus B$.
\end{prop}

\vspc
When $(g=A^{-1},\omega, T_0)$ satisfies the EIC, the restrictions of 
$U_c, U_{A,\omega}$ to $D_{g,\omega} \times \R$ are conormal distributions 
associated with the surface $t=\alpha_{g,\omega}$ and have progressing wave expansions, 
which play a crucial role in our proofs.
We define the objects needed for the progressing wave expansion.
On the region $D_{g,\omega}$, define the vector field $\T_{g,\omega}$ as
\begin{align*}
\T_{g,\omega} f :=   \la \nabla_g \alpha_{g,\omega}, \nabla_g f \ra_g.
\end{align*}
For $p \in P_\omega$, we know that $\gammadot_{p,g,\omega} = \nabla_g \alpha_{g,\omega}$ hence
\begin{equation}
( \T_{g,\omega} f)(\gamma_{p,g,\omega}(r)) = \frac{d}{dr} \left ( f( \gamma_{p,g,\omega}(r) ) \right ),
\qquad r < T_0.
\label{eq:transport}
\end{equation}
Also note that when $g=A^{-1}$ we have
\begin{align}
\T_{g,\omega} f = a_{ij} \, \pa_i \alpha_{g,\omega} \, \pa_j f.
\label{eq:transportA}
\end{align}
We use the following distributions on $\R$:
\[
K_{-3}(s) := \delta''(s), ~~ K_{-2}(s) := \delta'(s), ~~K_{-1}(s) := \delta(s), \qquad 
K_i(s) := \frac{s_+^i}{i!}, \qquad i=0,1,2, \cdots,
\]
and note that $\pa_s K_i = K_{i-1}$ for $i \geq -2$. 

When $(g,\omega, T_0)$ has the EIC, Proposition \ref{prop:ualpha} in the appendix gives detailed structural 
information about $U_c, U_{A,\omega}$, including their progressing wave expansions, obtained using a standard 
procedure.
This proposition is important for our proofs of the uniqueness results for the $c$ and $A$ problems, and for the 
construction of the crucial `complementary' solutions.

%%%%%%%%%%%%%%%%%%%% 

For the $c$ problem, with $A=I$ and $b=0$, \cite{rs20a} has a stability result but for the problem
with data associated with two different incoming waves - one traveling in the direction $\omega$ and the other traveling 
in the direction $-\omega$. These two solutions are `complementary' in a sense needed for the proof of the result.
In \cite{ms22}, this proof was adapted to the $c$ problem for the operator $\Box_g + c$ with $(\R^n,g)$ admissible, 
$\omega = e_n$, and
\[
g(x) = \begin{bmatrix} h(y) & 0 \\ 0 & 1 \end{bmatrix}
\]
for a Riemannian metric $(h(y), \R^{n-1})$, where the data was associated with sources consisting of
incoming plane waves from the directions $e_n$ and $-e_n$. For this $g$, note that $(g, \pm e_n, T)$ satisfies the EIC property for all $T$,
and $\alpha_{g,e_n} = x_n = \alpha_{g,-e_n}$. The {\bf single wave} fixed angle scattering problem for $c$, for the 
operator $\L$ or for $\Box_g + c$, remains unsolved even when $A=I$, $b=0$ or $g=g_{Eucl}$.

We prove our results for the $c$ and $A$ problems using a modification of the Bukhgeim-Klibanov method 
(see \cite{rs20a}, \cite{ms22}) which require
a solution $w \in C^2(\Bbar \times \R)$ of $\L w=0$ so that the unknown coefficient can be recovered from the 
(perhaps one sided) trace of $w$ on the surface $t=\alpha_{g,\omega}$. While the $U$ solutions in 
Proposition \ref{prop:ualpha} have the trace property, they fail to be $C^2$ on $\Bbar \times \R$ because of the singularity
on $t=\alpha_{g,\omega}$. So one needs data from a second solution $V$ (we call it the complementary solution) so that $U+V$ has the desired properties. We construct the complementary solution $V$ using a source constructed from the data from the $U$ solution, that is $U|_{\pa B \times \R}$ without knowing the coefficients of $\L$ on $B$.
We fix a positive integer $p > (n+1)/2 + 1$ to be used for the constructions of the complementary solutions 
$V_c, V_{A,\omega}$ for all the $c,A$.
\begin{prop}[The complementary solutions]\label{prop:complementary}
Suppose $\L$ is an admissible operator on $\R^n \times \R$ and 
$(g=A^{-1}, \omega, T_0)$ satisfies the EIC. Fix a positive integer $p > (n+1)/2 + 1$.
\vspc
\begin{enumerate}[(a)]
\item Let $U_c$ be the solution of the IVP \eqref{eq:Ucde}, \eqref{eq:Ucic}, and $u_c, f_{i}$, $i=-1, \cdots, p$ be
the smooth functions guaranteed by Proposition \ref{prop:ualpha}. Define the distribution
\[
\psi_c(x,t)= - f_{-1}(x) \delta(\alpha_{g,\omega}(x) -t) + \sum_{i=0}^p (-1)^i f_i(x) K_i( \alpha_{g,\omega}(x)-t) , 
\qquad (x,t) \in \pa B \times \R.
\]
Let $V_c$ be the unique solution\footnote{Please see the proof of the proposition regarding 
the existence and the uniqueness of the solution.} on $\Bbar \times \R$ of the final boundary value problem
\begin{equation}
\L V_c =0 ~ \text{on } B \times \R; \qquad V_c = \psi_c ~\text{on } \pa B \times \R;
\qquad V_c =0 ~ \text{on } B \times (T_0, \infty),
\label{eq:Vcde}
\end{equation}
of the form
\[
V_c(x,t) = - f_{-1}(x) \delta( \alpha_{g,\omega}(x) -t) + v_c(x,t) H(\alpha_{g,\omega}(x)-t), \qquad (x,t) \in \Bbar \times \R,
\]
with $v_c \in C^2(\Bbar \times \R)$. Then $w_c = U_c + V_c$ is in $C^2(\Bbar \times \R)$ and $\L w_c =0$ 
on $\Bbar \times \R$. Further,
\[
w_c(x,t) = \begin{cases} u_c(x,t), & \text{for } x \in \Bbar, ~ t \geq \alpha_{g,\omega}(x) \\ v_c(x,t), & 
\text{for } x \in \Bbar, ~ t \leq \alpha_{g,\omega}(x) \end{cases}.
\]
\item Let $U_{A,\omega}$ be the solution of the IVP \eqref{eq:UAde}, \eqref{eq:UAic},  and 
$u_{A,\omega}, f_{i}$, $i=0, \cdots, p$ be the smooth functions 
guaranteed by Proposition \ref{prop:ualpha}. Define the distribution
\[
\psi_{A,\omega}(x,t)=\sum_{i=0}^p (-1)^i f_i(x) \, K_i(\alpha_{g,\omega}(x) -t),  
\qquad (x,t) \in \pa B \times \R.
\]
Let $V_{A,\omega}(x,t)$ be the unique solution on $\Bbar \times \R$ of the final boundary value problem
\begin{equation}
\L V_{A,\omega} =0 ~ \text{on } B \times \R; \qquad V_{A,\omega} = \psi_{A,\omega}~\text{on } \pa B \times \R;
\qquad V_{A,\omega} =0 ~ \text{on } B \times (T_0, \infty).
\label{eq:VAde}
\end{equation}
of the form
\[
V_{A,\omega}(x,t) = v_{A,\omega}(x,t) H(\alpha_{g,\omega}(x)-t),
\qquad (x,t) \in B \times \R
\]
with $v_{A,\omega} \in C^2(\Bbar \times \R)$.
Then $w_{A,\omega} = U_{A,\omega} + V_{A.\omega}$ is in $C^2(\Bbar \times \R)$ and 
$\L w_{A,\omega}=0$ on $\Bbar \times \R$. Further
\[
w_{A,\omega}(x,t) = \begin{cases} u_{A,\omega}(x,t), & t \geq \alpha_{g,\omega}(x), ~ x \in \Bbar \\ 
v_{A,\omega}(x,t), & t \leq \alpha_{g,\omega}(x), ~ x \in \Bbar \end{cases}.
\]
\end{enumerate}
\end{prop}
\vspc
\noindent
Remarks:
\vspc
\begin{itemize}
\item As seen in the proof of Proposition \ref{prop:complementary}, the construction of $V_c$ and $V_{A,\omega}$ requires 
knowing $U_c$ and $U_{A,\omega}$ only on the region $\pa B \times (-2, T_0)$.
\item The $V_c, V_{A,\omega}$ are constructed by solving a final time IBVP and generate data for the inverse problem.
One may object to this because backward in time solutions are not `physical'. However $V_c(x,-t)$ and
$V_{A,\omega}(x,-t)$ are solutions of a `forward' IBVP and these are physical, so there is no issue.
\item In 2018, Matti Lassas communicated a construction of a $V_c$ similar to the one we use. 
The support property for the $V_c, V_{A,\omega}$ is new and crucial for our proofs.
\end{itemize}

\vspc
\noindent
We use the following surfaces and regions. Suppose $(g,\omega, T_0)$ satisfies the EIC and $T>T_0$.
Define
\begin{gather*}
\Sigma := \pa B \times [-T, T],
\qquad
\Sigma_+ := \Sigma  \cap \{ t \geq \alpha_{g,\omega}(x) \},
\qquad
\Sigma_- := \Sigma  \cap \{ t \leq \alpha_{g,\omega}(x) \},
\\
Q := \Bbar \times [-T,T], \qquad Q_+ = Q \cap \{ t \geq \alpha_{g,\omega}(x) \}, \qquad 
Q_- := Q \cap \{ t \leq \alpha_{g,\omega}(x) \},
\\
C := \{ (x, \alpha_{g,\omega}(x)) : x \in \pa B \} = \Sigma_+ \cap \Sigma_-, \qquad 
\Gamma := \{ (x, \alpha_{g,\omega}(x)) : x \in \Bbar \} = Q_+ \cap Q_-,
\\
H_{\pm} := \{ (x, \pm T): x \in \Bbar \}.
\end{gather*}
In the notation, we have suppressed the dependence on $T,\omega$ and $g$, as it will be clear from the context 
which $T,\omega,g$ are being used. 
For a function $w$ on a manifold $M$ and any non-negative integer $k$, $\|w\|_{k,M}$ will stand for one of the equivalent 
definitions of the Sobolev norm $\|w\|_{H^k(M)}$.

Suppose $\L$ is admissible, $(g=A^{-1},\omega,T_0)$ satisfies the EIC and $T \geq T_0$.
Define the coefficient to data forward maps 
\[
\G : c \to [U_c|_{\Sigma}, \pa_\nu v_c|_{\Sigma_-}],
\qquad 
\G: A \to [U_{A,\omega}|_\Sigma, \pa_\nu v_{A,\omega}|_{\Sigma_-}]_{\omega \in \Omega},
\qquad
\]
for some finite set $\Omega$ of unit vectors in $\R^n$.
Our goal is to study the injectivity and the stability of $ \G$. We assumed a known fixed $A,b$ when defining $\G(c)$,
and a known fixed $b,c$ when defining $\G(A)$.
We have the following uniqueness result for the $c$ problem.
%%%%%%%%
%
\begin{theorem}[Uniqueness for the $c$ problem]\label{thm:cunique}
Suppose $\L:= \pa_t^2 - a_{ij} \pa_i \pa_j - b_i \pa_i - c$ is admissible,
$\omega$ a unit vector in $\R^n$, $T_0>0$, and the metric $g=A^{-1}$ has the following properties:
\vspc
\begin{enumerate}[(a)]
\item there is a smooth strictly convex (w.r.t $g$) function $\kappa : \Bbar \to \R$,  with no 
critical points;
\item $(g,\omega,T_0)$ satisfies the EIC.
\end{enumerate}  
\vspc
If $c' \in C_c^\infty(B)$ and $T$ is large enough\footnote{We need 
$T > T_{*,g,\omega}$ where $T_{*,g,\omega}$ is the 
number guaranteed by Proposition \ref{prop:phiexistence}.} then
\[
 (U_{c} - U_{c'})|_\Sigma=0, \qquad (\pa_\nu v_c - \pa_\nu v_{c'})|_{\Sigma_-}=0
 \]
implies $c=c'$. Here $U_c, v_c, U_{c'}, v_{c'}$ are the solutions guaranteed by Propositions \ref{prop:uforward},
\ref{prop:complementary} for the operators $\L$ and $\L' = \pa_t^2 - a_{ij} \pa_i \pa_j - b_i \pa_i - c'$.
\end{theorem}
%%%%%%%%%%%%

It should be possible to modify our proof to obtain a Lipschitz stability result, using the exterior estimates on the region 
$(\R^n \setminus B) \times [-T,T]$ (where $\L = \L'= \Box$) as done in \cite{ms22}.
Conditions for the existence of a strictly convex function on a manifold $(\overline{B}, g)$ with strictly convex boundary (see (a) in Theorems \ref{thm:cunique}, \ref{thm:Aunique}) are discussed in \cite[Section 2]{paternain2019}. In particular, if $(\overline{B}, g)$ has nonpositive sectional curvature (or more generally no focal points), then the function $\kappa(x) = d_g(x,p)^2$ where $p$ is slightly outside $\overline{B}$ is strictly convex and has no critical points. Moreover, if $n=2$, then $(\overline{B}, g)$ admits a strictly convex function if and only if it is non-trapping in the sense that any geodesic reaches the boundary in finite time.

%%%%%%%%%% The $A$ problem %%%%%%%%%%%%%%

For our next result, we need a spanning condition, in addition to the 
EIC and the convexity condition required for Theorem \ref{thm:cunique}. 
\begin{defn}
A finite set $\{v_k\}_{k=1}^N$ of column vectors in $\R^n$  is said to have the {\em Spanning Condition} if
the span of  $ \{v_k \, v_k^T \}_{k=1}^N $ is the set of all $n \times n$ real symmetric matrices. 
\end{defn}

\vspc
\noindent
We now state our result for the $A$ problem; Propositions \ref{prop:eictransfer} is used implicitly in the statement of the result.
%%%%%%%%%%
\begin{theorem}[Uniqueness for the $A$ problem]\label{thm:Aunique}
Suppose $T_0>0$, $\L:= \pa_t^2 - a_{ij} \pa_i \pa_j - b_i \pa_i - c$, $\L':= \pa_t^2 - a_{ij}' \pa_i \pa_j - b_i \pa_i - c$ are 
admissible and let $A=(a_{ij}), A'= (a'_{ij})$.
Let $\Omega = \{ \omega_k \}_{k=1}^N$ be a set of unit vectors in $\R^n$ such that the metric $g=A^{-1}$ has the following properties:
\vspc
\begin{enumerate}[(a)]
\item there is a smooth strictly convex (w.r.t $g$) function $\kappa : \Bbar \to \R$ with no critical points;
\item $(g,\omega_k,T_0)$ satisfies the EIC for $k=1, \cdots, N$;
\item[(c)] $\{\nabla \alpha_{g,\omega_k}\}_{k=1}^N$ satisfies the Spanning Condition at each point of $\Bbar$. 
\end{enumerate}  
\vspc
If $T$ is large enough\footnote{ We need $T> 1 + T_0 + T_{*,g,\omega}$ for each $\omega \in \Omega$, where 
$T_{*, g, \omega}$ is the 
number guaranteed by Proposition \ref{prop:phiexistence}.} and 
\[
(U_{A',\omega_k} - U_{A,\omega_k})|_\Sigma =0,
\qquad (\pa_\nu v_{A',\omega_k} - \pa_\nu v_{A,\omega_k})|_{\Sigma_-} =0,
\]
for each $k=1, \cdots, N$, then $A'=A$.
\end{theorem}
\vspc
Our proof, combined with an exterior estimate (as in \cite{ms22}) and the additional assumption that
$(g=A^{-1}, T_2, \omega_k)$ and $(g'=A'^{-1}, T_2, \omega_k)$ satisfy the EIC for some $T_2>T$, should give a 
Lipschitz stability result. For the linearized problem, but for the recovery of the velocity (so $A = \rho^{-1}I$), there is an instability result in \cite{stefanov2013}. However, it does not contradict our stability claim.
As suggested in the proof of Theorem \ref{thm:Aunique}, for the stability result, the
boundary data for $\L$ and $\L'$ are compared at the same $x$ but with a shift in $t$ in the data for $\L'$,
whereas in \cite{stefanov2013} the boundary data for $\L,\L'$ are compared at the same $x,t$. Our approach to 
measuring the difference in the data seems appropriate, as a 
small change in the velocity can result in the same wave arriving at the boundary at a slightly later time and,
even for small $\tau$,  $\| f(\cdot) - f(\cdot - \tau)\|_{L^2}$ can be of the same order as $\|f(\cdot)\|_{L^2}$, if $f$ has high 
frequency components.

Here are some remarks about the hypotheses of Theorem \ref{thm:Aunique}. 
\vspc
\begin{itemize}
\item Note that $A'$ is not required to satisfy any of the conditions (a), (b), (c) of Theorem \ref{thm:Aunique}. 
\item Condition (c) is motivated by the case where $A=I$ (hence $g=I$) and 
\[
\Omega = \{e_i : i=1, \cdots, n\} \cup \{(e_i + e_j)/\sqrt{2} : i,j=1, \cdots,n, ~ i \neq j\}.
\]
Then, for each $\omega \in \Omega$, $\alpha_{\omega} = x \cdot \omega$ so 
$\nabla \alpha_{\omega} = \omega$, and one can check that the set $\Omega$ satisfies the Spanning Condition.
\item In \cite{ors26d}, we provide sufficient conditions under which the (a), (b), (c) in 
Theorem \ref{thm:Aunique} are satisfied. In particular, these conditions hold if 
$\| g - g_{\mathrm{Eucl}}\|_{C^3(\overline{B})}$ is sufficiently small.
\item To simplify notation, we use the symbol $\alpha_k$ for $\alpha_{g,\omega_k}$.
One can show (see the proof of Theorem \ref{thm:Aunique}) that condition (c) of Theorem \ref{thm:Aunique}
implies
\begin{equation}
\sum_{k=1}^N | (\nabla \alpha_k)^T (A'-A) (\nabla \alpha_k)(x)|^2 \cgeq \|(A'-A)(x)\|^2_\infty,
\qquad x \in \Bbar,
\label{eq:alphaAAp}
\end{equation}
for any two real symmetric matrix functions $A(x),A'(x)$, with the constant independent of $x,A(x),A'(x)$. 
This is the only implication of condition (c) that plays a role in the proof of Theorem \ref{thm:Aunique}.
Here 
\[
P(x) \cgeq Q(x) \qquad x \in \Bbar
\]
means 
\[
P(x) \geq C Q(x), \qquad x \in \Bbar
\]
for some constant $C$ independent of $x \in \Bbar$.
\item If we drop the requirement (c) from Theorem \ref{thm:Aunique}, we no longer have a result for arbitrary $A'$ but a 
more restricted result is still true. A minor modification of the proof of Theorem \ref{thm:Aunique} shows
that if $\Omega=\{\omega\}$ for some unit vector $\omega$ and conditions (a), (b) hold and we do not require (c),
then $\G(A') = \G(A)$ implies $A'=A$ for any $A'$ of the form
\[
A'(x) := A(x)+ \mu(x) B(x), \qquad x \in \Bbar
\]
where $\mu$ is a smooth real valued function on $\Bbar$, and $B(x)$ is a smooth positive definite matrix valued
function on $\Bbar$. The modification in the proof needed is the observation
\begin{align*}
| (\nabla \alpha)^T (A'-A) (\nabla \alpha)(x)| & = |\mu(x)|  \, | (\nabla \alpha)^T B (\nabla \alpha)(x) |
\cgeq |\mu(x)| \, |\nabla \alpha(x)|^2 \cgeq |\mu(x)|
\\
& \cgeq \| (A'-A)(x)\|_\infty, \qquad x \in \Bbar,
\end{align*}
with the constant independent of $x$. In particular, taking $A(x)= \rho(x)^{-1} I$, $A'(x) = \rho'(x)^{-1}I$ for positive smooth functions $\rho, \rho'$, gives a uniqueness result for the $\rho$ recovery problem for the operator $\rho \pa_t^2 - \Delta$,
provided $\Omega=\{\omega\}$ for some unit vector $\omega$ and the metric $\rho I$ satisfes (a), (b) for this $\Omega$
- condition (c) is not needed.
\end{itemize}

%%%%%%%%%
%%%%%%%%%%
Our recent articles \cite{ors26a, ors26b} have a comprehensive review of the literature associated with the problems 
studied in this article. Since the publication of our articles, the articles \cite{ali2025, ak2025} have appeared 
and address the two formally determined inverse problems - the $c$ recovery problem for the operator $\Box_g + c$ and
the $g$ recovery problem for the operator $\Box_g$. \cite{ali2025} obtains results for the one space dimensional case 
and \cite{ak2025} obtains results for the multidimensional case.
The problems are different from ours in that the forward problem
is an IBVP on the infinite cylinder $D \times [0,\infty)$ for some smooth bounded domain $D$ in $\R^n$, with Dirichlet boundary condition, and the source being the initial value which may be (surprisingly) an {\bf unknown} function
but whose eigenfunction expansion (in terms of the Dirichlet eigenfunctions of $\Delta_g$ on $D$) contains all the 
components, and the data is the value of the solution on $\Omega \times [0,\infty)$ for some open subset 
$\Omega$ of $D$.
%%%%%%
%%%%%%%

For a real $\sigma$, a submanifold $M$ of $\R^n \times \R$, and a weight $\vphi \in C^\infty(M)$, we define the following norms of  smooth enough functions $f: M \to \R$ :
\[
\| f \|_{0,M,\sigma} := \left ( \int_M e^{2 \sigma \vphi} |f|^2 \, dS \right )^{1/2},
\qquad
\| f \|_{1,M,\sigma} := \left ( \int_M e^{2 \sigma \vphi} ( |\nabla_M f|^2 + \sigma^2 |f|^2 ) \, dS \right )^{1/2},
\]
where $\nabla_M$ represents a finite set of vector fields on $M$ spanning its tangent bundle. Note that choosing different 
$\nabla_M$ results in equivalent norms, if $M$ is compact. We also define (note the missing $\sigma$ and $\vphi$)
\[
\| f \|_{0,M} := \left ( \int_M |f|^2 \, dS \right )^{1/2},
\qquad
\| f \|_{1,M} := \left ( \int_M ( |\nabla_M f|^2 + |f|^2 ) \, dS \right )^{1/2}.
\]
For functions $f_1, f_2, \cdots, f_m$ on $M$ and $k=0,1$, we define
\[
\|[f_1, \cdots, f_m]\|_{k,\sigma,M} := \sum_{i=1}^m \|f_i\|_{k,\sigma,M},
\qquad
\|[f_1, \cdots, f_m]\|_{k,M} := \sum_{i=1}^m \|f_i\|_{k,M}.
\]

The following proposition is crucial in the proofs of both the theorems in this article.
\begin{prop}\label{prop:aux} 
Suppose $\L := \pa_t^2 - a_{ij}\pa_i \pa_j - b_i \pa_i - c$ is an admissible operator on 
$\R^n \times \R$, $(g=A^{-1},\omega,T_0)$ satisfies the EIC, $T>T_0$, and $\vphi(x,t)$ is a smooth
strongly pseudoconvex\footnote{See the appendix of \cite{rs20a} for the definition.} 
(w.r.t $\L$) function on $Q$ with
\[
\min_\Gamma \vphi -  \max_{H_{\pm}} \vphi \geq 2 \delta,
\]
for some $\delta>0$.
There are $\sigma_0>0, k>0$ such that, for all $w \in C^2(Q)$ and all $\sigma \geq \sigma_0$, we have
\begin{align}
\| w\|^2_{1,\sigma,\Gamma} \cleq \| \L w \|^2_{0,\sigma, Q}  + \sigma^2 e^{2 \sigma \vphi_H}
\|\L w\|^2_{0,Q} +
e^{k \sigma} ( \|w\|^2_{1, \Sigma} +  \| \pa_\nu w \|^2_{0, \Sigma});
\label{eq:auxest}
\end{align}
here $\vphi_H := \max_{H_{\pm}} \vphi$. 
The quantities $\sigma_0, k$ and the inequality constant are independent of $w$ and $\sigma$.
%  and determined by any upper bound on $\|A\|_{C^2}$, $\|b\|_{C^2}$, $\|c\|_{C^2}$, 
% $\|\vphi\|_{C^2(Q)}$, $\delta$, $T$ and $\|\alpha_{g,\omega}\|_{C^1(B)}$.
\end{prop}
\vspc
This proposition is in \cite{ms22} in some form and is a generalization of a similar estimate for the $A=I$ case in 
\cite{rs20b}. Both of these are adaptions, to the case where $\Gamma$ is a characteristic surface of $\L$,
of a similar estimate in \cite{BY17} with $\Gamma$ replaced by  $\Bbar \times \{t=0\}$. 
Since this proposition is not explicitly stated in
\cite{ms22} and its proof there consists of combining pieces from the proofs of several propositions, 
we provide a compact proof of this proposition in the appendix.
%%%%%%%%%%%%%%%
%%%%%%%%%%%%%%%

The proofs of the theorems in this article use Proposition \ref{prop:aux} and need the
weight $\vphi$ constructed using the convex function guaranteed by the hypotheses of these theorems. 
The following proposition constructs these $\omega$ dependent weights explicitly. The construction shows 
(see \eqref{eq:vphidef}) that the 
trace of these weights on the surfaces $t=\alpha_{g,\omega}$ is a function of {\bf $x$ but independent of 
$\omega$}  - this will be crucial for the proof of Theorem \ref{thm:Aunique}.

\begin{prop}\label{prop:phiexistence}
Suppose $(\R^n, g)$ is an admissible Riemannian metric, $\omega$ a unit vector in $\R^n$, and $T_0>0$, 
with the following geometrical properties:
\vspc
\begin{itemize}
\item there is a smooth strictly convex function (w.r.t the Riemannian metric $g$) $\kappa:\Bbar \to \R$
with no critical points in $\Bbar$, 
\item $(g,\omega,T_0)$ satisfies the EIC. 
\end{itemize}
\vspc
There is a constant $T_* > T_0$ so that, for any $T \geq T_*$, there are $\ep, \lambda_*$, so that for any 
$\lambda \geq \lambda_*$,
\begin{equation}
\vphi(x,t) = e^{ \lambda \left ( \|\kappa\|_\infty + \kappa(x) - \ep (t-\alpha_{g,\omega}(x))^2 \right )}
 \qquad (x,t) \in \Bbar \times \R,
\label{eq:vphidef}
\end{equation}
has  the following properties:
\vspc
\begin{enumerate}[(a)]
\item $\vphi$ is strongly pseudoconvex\footnote{See the appendix of \cite{rs20a} for the definition.}  
with respect to\footnote{Any operator with the same principal symbol as $\Box_g$.} $\Box_g$ on the 
region $\Bbar \times [-T,T]$;
\item ${\displaystyle \inf_\Gamma \vphi > \sup_{H_{\pm}} \vphi }$;
\item ${\displaystyle \lim_{\sigma \to \infty} \sigma \, \eta(\sigma)=0}$ where
\[
\eta(\sigma) := \sup_{x \in \Bbar} \int_{-T}^T e^{2 \sigma( \vphi(x,t) - \vphi(x,\alpha_{g,\omega}(x)) )}\, dt.
\]
\end{enumerate} 
If $\|\alpha_{g,\omega}\|_{C^2(\Bbar)} \leq M$, then $T_*$ is determined by $g, \kappa, M, T_0$. Further, 
$\ep$ is determined by $g,\kappa,T, M$, and $\lambda_*$ is determined by $g,\kappa,T, \ep, \omega$.
\end{prop}
%%%%%%%%%
%
\noindent
The weight $\vphi$ is strongly pseudoconvex also for $\L$ with $g=A^{-1}$. Our proposition is essentially 
\cite[Lemma 3.5]{ms22}, except we prove a stronger decay rate for $\eta(\sigma)$ and
track carefully what $T_*, \ep, \lambda_*$ depend on because, for the proofs of the theorems, we need to choose 
$T,\ep,\lambda$ so that the same values work for a finite number of $\omega$. The proof of the proposition
actually shows $\lim_{\sigma \to \infty} \sigma^k \eta(\sigma)=0$ for any real $k$, but we need only the
$k=1$ case in this article. The proof is in the appendix.
%

%%%%%%%%%%%%%%%%%
%%%%%%%%%%%%%%%%%

The rest of the article consists of the following. 
\vspc
\begin{itemize}
\item Section \ref{sec:pairsoln} contains the proof of Proposition 
\ref{prop:complementary} - the construction of the complementary solutions 
$V_c, V_{A,\omega}$ and the regularity of the solutions $w_c, w_{A,\omega}$ 
constructed from the $U$ and the $V$ solutions. 
\item Section \ref{sec:cunique} contains the proof of Theorem \ref{thm:cunique} - the injectivity result for the $c$ problem.
\item Section \ref{sec:Aunique} contains the proof 
of Theorem \ref{thm:Aunique} - the injectivity result for the $A$ problem.
\end{itemize}
\vspc
The appendix contains the proofs and calculations for various auxilliary results, needed in the proofs of the main theorems. 
The proofs of these auxiliary results are based mostly on ideas already in the literature, though some new ideas were 
needed. The appendix consists of the following.
\vspc
\begin{itemize}
\item Subsection \ref{subsec:uforward} contains the proof of Proposition \ref{prop:uforward} about 
the existence, uniqueness and the properties of $U_c, U_{A,\omega}$.
\item Subsection \ref{subsec:eictransfer} contains the proof of Proposition \ref{prop:eictransfer} about the
transfer of the EIC from one metric to another.
\item Subsection \ref{subsec:progressing} constructs the progressing wave expansions of 
$U_c,U_{A,\omega}$ when $(g=A^{-1},\omega,T_0)$ has the EIC. These expansions are given in Proposition \ref{prop:ualpha} and the support property is new.
\item Subsection \ref{subsec:aux} contains the proof of Proposition \ref{prop:aux} - a proposition used in the proofs of all
the theorems in this article.
\item Subsection \ref{subsec:weight} contains the proof of Proposition \ref{prop:phiexistence} - the
construction of the special Carleman weight needed in the hypothesis of Proposition \ref{prop:aux}. 
\item Subsection \ref{subsec:conseq} contains a proposition which captures the essential idea at the end 
of the proofs of Theorem \ref{thm:cunique} and Theorem \ref{thm:Aunique}. 
\item Subsection \ref{subsec:standard} contains some standard but detailed calculations used in the proofs of the propositions and theorems. Including them in the proofs of the theorems and propositions 
would distract from the main ideas in those proofs.
\end{itemize}

\noindent
{\bf Acknowledgements.} 
L.O.~was supported by the European Research Council of the European
Union, grant 101086697 (LoCal). L.O.~and M.S.~were partly supported by the Research Council
of Finland, grants 353091 and 353096 (Centre of Excellence in Inverse Modelling and Imaging),
359182 and 359208 (FAME Flagship) as well as 347715. Rakesh’s work was partly funded by grants
DMS 1908391 and DMS 2307800 from the National Science Foundation of USA. Views and opinions
expressed are those of the authors only and do not necessarily reflect those of the European Union or
the other funding organizations. Neither the European Union nor the other funding organizations
can be held responsible for them

%%%%%%%%%%%%%%%%%%%
%%%%%%%%
\section{The construction of $V_c$ and $V_{A,\omega}$}\label{sec:pairsoln}

We prove the part of Proposition \ref{prop:complementary} about the construction of $V_c$. 
We do not give the construction of $V_{A,\omega}$ as it is
almost identical to the construction of $V_c$. Since we work with a fixed $A,b,c$, and a fixed $g=A^{-1}$, and a fixed
$\omega$, to keep the notation simple, we do not display the dependence on $A,b,c, g$ or $\omega$. So
we use $U,u,\alpha, D$ instead of $U_c, u_c, \alpha_{g,\omega}, D_{g,\omega}$. Since $(g=A^{-1},\omega,T_0)$ 
satisfies the 
EIC, Proposition \ref{prop:exinj} guarantees that $\alpha$ is smooth on $D$, in particular on
a neighborhood of $\Bbar$. 

First we show that $\alpha|_{\pa B}$ and $f_i|_{\pa B}$ for $i=-1,0, \cdots, p$ may be extracted from
$U|_{\pa B \times (-2, T_0)}$. From Proposition \ref{prop:uforward}, WF$(U|_{D \times (-\infty, T_0)}) = 
\Lambda|_{D \times (-\infty, T_0)}$, so the singular support of $U|_{\pa B \times (-2,T_0)}$ is $C$. Hence one can extract
$\alpha|_{\pa B}$ from $U|_{\pa B \times (-2, T_0)}$. 

Pick an $m > (n+1)/2 + p + 1$. From Proposition \ref{prop:ualpha},
\[
U(x,t) = f_{-1}(x) \delta(t-\alpha(x)) + \sum_{i=0}^m f_i(x) K_i (t-\alpha(x)) + R_m(x,t),
\qquad (x,t) \in D \times \R.
\]
with $f_i \in C^\infty(D)$, $R_m \in H^{m+1}_{loc}( D \times \R) \subset C^{p+2}(D \times \R)$, 
and $R_m(x,t)=0$ for $t \leq \alpha(x)$ because $U$ is supported in the region $ t \geq \alpha(x)$. 
Let $\rho \in C^\infty(\R)$ with
\[
\rho(s) = \begin{cases} 1, & |s| \leq 1, \\ 0, & |s| \geq 2 \end{cases}
\]
and let $\phi(x)$ be any smooth function on $\pa B$. Then, for any small $\ep >0$, 
\begin{align*}
\left \la U|_{\pa B \times \R}, \rho( (t- \alpha(x))/\ep ) \, \phi(x) \right \ra
& = \int_{\pa B} f_{-1}(x) \, \phi(x) \, dS_x + \int_{\pa B} \int_{\alpha(x) - 2 \ep}^{\alpha(x) + 2 \ep} \phi_\ep(x,t) \, dt \, dx,
\end{align*}
where $\phi_\ep$ is a bounded function on $\pa B \times \R$ with a bound independent of 
$\ep$. So
\[
\lim_{\ep \to 0_+} \left \la U|_{\pa B \times (-1, T_0)}, \rho( (t- \alpha(x))/\ep ) \, \phi(x) \right \ra
= \int_{\pa B} f_{-1}(x) \, \phi(x) \, dS_x,
\]
for every $\phi \in C^{\infty}(\pa B)$. Hence one can extract $f_{-1}|_{\pa B}$ from 
$U|_{\pa B \times (-2, T_0)}$.  
Further, we can recursively extract the values of $f_i$ on $\pa B$, for $i=0, \cdots, p$, since for $x \in \pa B$,
\[
f_i(x) = \lim_{t \to \alpha(x)_+} i! \, (t-\alpha(x))^{-i} \, [ \, U|_{\pa B \times (-2,T_0)} - \sum_{j=-1}^{i-1} f_j(x) 
K_j(t-\alpha(x) ) \,],
\qquad i=0,1, \cdots,p, 
\]
because $R_m \in C^{p+2}( D \times \R)$ and $R_m$ is zero for $t < \alpha(x)$.
We now show the existence and uniqueness of the solution $V$, of the IBVP \eqref{eq:Vcde}, of the form
\[
V(x,t) =  - f_{-1}\delta(\alpha(x)-t) + v(x,t) H(\alpha(x)-t), \qquad (x,t) \in B \times \R
\]
with $v \in C^2(\Bbar \times \R)$. 

We first prove the uniqueness. If $\Vbar$ is the difference of two such 
solutions of the IBVP \eqref{eq:Vcde}, then 
\[
\Vbar(x,t) = \vbar(x,t) H(\alpha(x)-t), \qquad (x,t) \in B \times \R
\]
with $\vbar(x,t) \in C^2(\Bbar \times \R)$ and
\begin{align}
\L \Vbar =0 ~ \text{on } B \times \R, \qquad \qquad \Vbar=0 ~ \text{on } \pa B \times \R.
\end{align}
Using calculations similar to the one in Subsection \ref{subsec:standard}, and that $\alpha(x)$ is the solution of
the eikonal equation for $g$, one can show that
\[
\L \Vbar = (\L \vbar) H(\alpha(x)-t) - [2 \pa_t + 2 a_{ij} (\pa_i \alpha) \pa_j + b_i \pa_i \alpha] \vbar(x,t) \, \delta(\alpha(x)-t).
\]
Hence
\begin{align}
\L \vbar =0  & \qquad \text{on } (B \times \R) \cap \{ t < \alpha(x) \}
\label{eq:compvbarde}
\\
[2 \pa_t + 2 a_{ij} (\pa_i \alpha) \pa_j + b_i \pa_i \alpha] \vbar(x,t) = 0
&  \qquad \text{on } \{ (x, \alpha(x)) : x \in B \}.
\label{eq:compvbarcc}
\\
\vbar(x,t) = 0 & \qquad \text{on } \{ (x,t) : x \in \pa B, ~ t \leq \alpha(x) \}.
\label{eq:compvbarbc}
\end{align}
Now the first two terms in the operator in \eqref{eq:compvbarcc} represent differentiation along the geodesics 
$\gamma_{p,\omega}$ for different $p \in P_\omega$ (see the proof of Lemma \ref{lemma:energy}), hence \eqref{eq:compvbarcc} represents a homogeneous linear first 
order ODE for the function $\vbar(x, \alpha(x))$, $x \in \Bbar$. Since $\vbar(x, \alpha(x)) =0$ for $x \in \pa B$, we can
conclude that
$\vbar(x,\alpha(x)) =0$ for $x \in \Bbar$. Therefore, on the region $ \{ (x,t) : x \in \Bbar, ~ t \leq \alpha(x) \}$, $\vbar$ is a 
$C^2$ solution of the characteristic BVP
\begin{align*}
\L \vbar=0 & \qquad \text{on } (B \times \R) \cap \{ t < \alpha(x) \}
\\
\vbar =0, & \qquad \text{on } \{ (x, \alpha(x)) : x \in \Bbar \}
\\
\vbar=0, & \qquad \text{on } \{ (x,t) : x \in \pa B, ~ t < \alpha(x) \}.
\end{align*}
Using energy estimate arguments similar to those in the proof of Lemma \ref{lemma:energy}, combined with Gronwall's 
inequality, one can show that $\vbar=0$ on the region $(B \times \R) \cap \{ t < \alpha(x) \}$. Hence $\Vbar=0$.

We will seek $V$ in the form
\begin{align*}
V(x,t) & =  V_p(x,T) + R(x,t), \qquad (x,t) \in B \times \R
\end{align*}
for a distribution $R(x,t)$ on $B \times \R$, where
\begin{align*}
V_p(x,t) & = - f_{-1}(x) \delta(\alpha(x)-t) + \sum_{i=0}^p (-1)^i f_i(x) \, K_i(\alpha(x) - t),
\qquad (x,t) \in B \times \R.
\end{align*}
Repeating the calculation in the construction of the $f_i$ in the proof
of Proposition \ref{prop:ualpha}, one can show that
\begin{align*}
\L V_p & = (-1)^p K_p(\alpha(x)-t) \, \calE f_p(x), \qquad (x,t) \in B \times \R.
\end{align*}
Hence
\begin{align*}
\L V & =  (-1)^p (\calE f_p) K_p(\alpha(x)-t) + \L R, \qquad \text{on } B \times \R,
\end{align*}
so $V$ will solve \eqref{eq:Vcde} iff $R$ is the solution of the final BVP
\begin{subequations}
\begin{align}
\L R = (-1)^{p+1} K_p(\alpha(x)-t) \, \calE f_p(x), & \qquad \text{on } B \times \R,
\label{eq:RVde}
\\
R =0, & \qquad \text{on } \pa B \times \R,
\label{eq:RVbc}
\\
R =0, & \qquad \text{on }  B \times (T_0, \infty).
\label{eq:RVic}
\end{align}
\end{subequations}
Now $K_p(\alpha(x) -t) \in H^p_{loc}(\R^n \times \R)$ and is supported in the region $t<T_0$ for $x \in \Bbar$.
So, by \cite[Theorem 2.2]{lasiecka1986}, 
the backward 
IBVP \eqref{eq:RVde} - \eqref{eq:RVic} has a solution, which lies in $H^{p+1}_{loc}(B \times \R)$; hence
$R \in H^{p+1}_{loc}(B \times \R) \subset C^2(\Bbar \times \R)$. Also $R=0$ on $\pa B \times \R$, 
$R=0$ for $t>T_0$, and the RHS of of \eqref{eq:RVde} is supported in $t \leq \alpha(x)$, so 
repeating the arguments in the proof of Lemma \ref{lemma:energy}, one can conclude that 
$R=0$ in the region $ t \geq \alpha(x)$. 

Define
\[
v(x,t) =  v_p(x,t) + R(x,t), \qquad (x,t) \in B \times \R
\]
where
\[
v_p(x,t) = \sum_{i=0}^p (-1)^i f_i(x) \, \frac{(\alpha(x)-t)^i}{i!} , \qquad (x,t) \in B \times \R.
\]
Then $v \in H^{p+1}_{loc}(B \times \R) \subset C^2(\Bbar \times \R)$ and
\[
V(x,t) = - f_{-1}(x) \delta(\alpha(x)-t) + v(x,t) H(\alpha(x)-t), \qquad (x,t) \in B \times \R.
\]

Now $w_c = U_c + V_c$ on $B \times \R$ so $\L w_c =0$ on $B \times \R$. Further, for $(x,t) \in B \times \R$,
\begin{align*}
w_c(x,t) & = \sum_{i=0}^p f_i(x) \, K_i(t - \alpha(x) + \sum_{i=0}^p (-1)^i f_i(x) K_i(\alpha(x)-t) + R_p(x,t) + R(x,t)
\\
& = \sum_{i=0}^p f_i(x) \frac{ (t-\alpha(x))^i}{i!} + R_p(x,t) + R(x,t),
\end{align*}
hence $w_c \in H^{p+1}_{loc}(B \times \R) \subset C^2(\Bbar \times \R)$. Finally, it is clear that $w_c = u_c$ on $t > \alpha(x)$ and $w_c=v_c$ on $t < \alpha(x)$.
%%%%%%%%%%%%%%%%%%%%%%%%%%%%%

\section{Proof of Theorem \ref{thm:cunique}}\label{sec:cunique}

In our notation, we drop the dependence on $g,\omega$ since only one $g=A^{-1}$ and one $\omega$ is 
used in the statement of the theorem. So $\alpha_{g,\omega}$ will be denoted by $\alpha$. Also, objects dependent on
$c'$, such as $U_{c'}, u_{c'}, v_{c'}, w_{c'}$ will be written as $U', u', v', w'$, and objects associated with $c$ such as
$U_c,u_c, v_c, w_c$ will be written as $U,u,v,w$.

Let $\L = \pa_t^2 - a_{ij} \pa_i \pa_j - b_i \pa_i - c$ and $\L' = \pa_t^2 - a_{ij} \pa_i \pa_j - b_i \pa_i - c'$, 
$g=A^{-1}$, and $u,v,w, \alpha, f_0$ ~ $u',v',w', f_0'$ the 
functions guaranteed by the first parts of Proposition \ref{prop:ualpha} and Proposition \ref{prop:complementary} 
for $\L,\L'$. 
Since the principal parts of $\L, \L'$ are identical, the same $g$ is the relevant Riemannian metric for $\L, \L'$,
so the same $\alpha$ is the relevant solution of the eikonal equation for $\L, \L'$.

Noting the hypothesis of Theorem \ref{thm:cunique}, from Proposition \ref{prop:aux}, we can find a $T_*$
so that for each $T > T_*$ there is a $\vphi$ on $Q$ satisfying the conclusion of Proposition \ref{prop:aux} 
for the operator $\L$. So, this $\vphi$ fulfills the hypothesis of Proposition \ref{prop:phiexistence}. We fix a $T > T_*$ 
and then fix the $\vphi$.

By hypothesis $U=U'$ on $\Sigma$, $\L = \L' = \Box$ on $(\R^n \setminus B) \times (-\infty, T)$, 
so by \cite[Proposition A.1]{ors26a}, we have $U=U'$ on $(\R^n \setminus \Bbar) \times (-\infty, T)$. However, for $(x,t) \in D_{g,\omega} \times \R$,
we have
\begin{align*}
U(x,t)  & = f_{-1} \delta(t-\alpha(x)) + u(x,t) H(t-\alpha(x)),
\\
U'(x,t)  & = f_{-1} \delta(t-\alpha(x)) + u'(x,t) H(t-\alpha(x)),
\end{align*}
for some $u,u' \in C^\infty(D_{g,\omega} \times \R)$. Hence $U=U'$ on $(\R^n \setminus \Bbar) \times (-\infty, T)$ 
implies $u=u'$, $\pa_\nu u = \pa_\nu u'$ on $\Sigma_+$ and the $\psi, \psi'$ in Proposition 
\ref{prop:complementary}, corresponding to $\L, \L'$, are identical - we need $U,U'$ only on $\pa B \times (-2,T_0)$
to construct the $\psi, \psi'$. Hence, by hypothesis, we also have
$v=v'$ and $\pa_\nu v = \pa_\nu v'$ on $\Sigma_-$.

Define 
\[
\wbar := w - w', \qquad \cbar := c - c';
\]
we have
\[
\wbar|_\Sigma=0, \qquad \pa_\nu \wbar|_\Sigma =0,
\]
and
\[
\L \wbar = \cbar \, w', \qquad \text{on } Q.
\]
Since $\wbar \in H^2(Q)$, $w' \in C^2(Q)$ - this is important, and $\wbar|\Sigma=0, \pa_v \wbar|_\Sigma=0$,
applying Proposition \ref{prop:aux} to $\vphi$ and $\wbar$, we have
\begin{equation}
\|\nabla_\Gamma \wbar \|_{1,\Gamma,\sigma}^2 \cleq  \|\cbar\|^2_{0,Q,\sigma}
+ \sigma^2 e^{2 \sigma \phi_H} \| \cbar \|^2_{0,Q};
\label{eq:carltemp}
\end{equation}
we have made use of the fact that $w' \in C^2(Q)$.

In Proposition \ref{prop:ualpha} we observe that $f_{-1}$ and $\T$ are the same for $c$ and $c'$, and 
$\calE-c= \calE'-c'$. Further
\[
(2 \T + (\calE -c) \alpha)f_0 = \calE f_{-1},
~~
(2 \T + (\calE -c) \alpha)f_0' = \calE' f_{-1},
\qquad \text{on } \Bbar,
\]
hence
\[
(2 \T + (\calE -c) \alpha)(f_0 - f_0')= (\calE -\calE') f_{-1} = \cbar f_{-1},
\qquad \text{on } \Bbar.
\]
Now
\[
\wbar=w-w'=u-u' = f_0 - f_0', \qquad \text{on } \Gamma.
\]
Since $(2 \T + (\calE -c) \alpha)$ is a vector field on $\Gamma$, \eqref{eq:carltemp} implies that
\begin{align*}
\int_\Bbar e^{2 \sigma \vphi(x,  \alpha(x))} \, |\cbar(x) \, f_{-1}(x)|^2 \, dx
& \cleq  \|\cbar\|^2_{0,Q,\sigma}
+ \sigma^2 e^{2 \sigma \phi_H} \| \cbar\|^2_{0,Q}.
\end{align*}
Now \eqref{eq:f1cde} is a linear first order homogeneous ODE on the geodesic $r \to \gamma_p(r)$ 
and has a positive initial value, so $f_{-1}$ is positive on $\Bbar$, hence has a positive lower bound independent of $c$.
Therefore
\begin{align}
\int_\Bbar e^{2 \sigma \vphi(x,  \alpha(x))} \, |\cbar(x)|^2 \, dx
& \cleq  \|\cbar\|^2_{0,Q,\sigma}
+ \sigma^2 e^{2 \sigma \phi_H} \| \cbar\|^2_{0,Q}.
 \label{eq:alphaTmax}
\end{align}

Our $\vphi$ was constructed to have the properties guaranteed by Proposition \ref{prop:phiexistence}.
We use the notation in Proposition \ref{prop:phiexistence}. Define $\delta$ by
\[
2 \delta := \inf_{x \in \Bbar} \vphi(x,\alpha(x)) - \vphi_H;
\]
then $ \delta>0$ by (b) of Proposition \ref{prop:phiexistence}. Now
\begin{align*}
\|\cbar\|_{0,Q,\sigma}^2 & = \int_Q e^{2 \sigma \vphi} |\cbar|^2
= \int_B e^{2 \sigma \vphi(x,\alpha(x))} |\cbar(x)|^2 \int_{-T}^T e^{2 \sigma (\vphi(x,t) - \vphi(x,\alpha(x)))} \, dt \, dx
\\
& \leq \eta(\sigma) \int_\Bbar e^{2 \sigma \vphi(x,  \alpha(x))} \, |\cbar(x)|^2 \, dx
\end{align*}
and
\begin{align*}
e^{2 \sigma \vphi_H} \|\cbar\|_{0,Q}^2 
= 2 T e^{2 \sigma \vphi_H} \int_B |\cbar|^2
\cleq  e^{-2 \sigma \delta} \int_B e^{2 \sigma \vphi(x,\alpha(x))} |\cbar(x)|^2.
\end{align*}
Using these in \eqref{eq:alphaTmax} we obtain
\[
\int_\Bbar e^{2 \sigma \vphi(x,  \alpha(x))} \, |\cbar(x)|^2 \, dx
\cleq \left (  \eta(\sigma) + \sigma^2 e^{-2 \sigma \delta} \right ) 
\int_\Bbar e^{2 \sigma \vphi(x,  \alpha(x))} \, |\cbar(x)|^2 \, dx,
\]
for large enough $\sigma$. Since $\lim_{\sigma \to \infty} \eta(\sigma)=0$, taking $\sigma$ large enough, we 
obtain
\[
\int_\Bbar e^{2 \sigma \vphi(x,  \alpha(x))} \, |\cbar(x)|^2 \, dx \leq 0.
\]
Hence $\cbar=0$.

%%%%%%%%

%%%%%%%%%%%%%%%%%%
%%%%%%%%%%%%%%%%%%%%

\section{Proof of Theorem \ref{thm:Aunique}}\label{sec:Aunique}

We associate the metric $g=A^{-1}$ with $\L$ and $g'=A'^{-1}$ with $\L'$. 
Objects associated with $\L', A'$ and $\omega_k$, such as $U_{A',\omega_k}, \alpha_{g',\omega_k}, 
w_{A',\omega_k}$ will be written as $U'_k, \alpha'_k, w'_k$, and objects associated with $\L,A$ and $\omega_k$, 
such as $U_{A,\omega_k}, \alpha_{g,\omega_k}, w_{A,\omega_k}$ will be written as $U_k, \alpha_k, w_k$. Also 
define the elliptic operators
\[
\calE := a_{ij} \pa_i \pa_j + b_i \pa_i + c, \qquad \calE' := a'_{ij} \pa_i \pa_j + b_i \pa_i + c.
\]

Since $\calE$ is elliptic on $\Bbar$, the level surfaces of $\kappa(x)$ (of any function with no critical points) are 
pseudo-convex w.r.t $\calE$ on $\Bbar$. Hence, from \cite[Proposition A.5]{rs20b}, there is a $\lambda_0$ such that for
any $\lambda \geq \lambda_0$,  $e^{\lambda \kappa(x)}$ is strongly pseudo-convex w.r.t $\calE$ on $\Bbar$.

For $g=A^{-1}$ and for $k=1, \cdots, N$, we have a $T_{*,g,\omega_k}$ guaranteed by Proposition \ref{prop:phiexistence}.
Define
\[
T_* = \max_{k=1, \cdots,n} T_{*,g,\omega_k},
\]
and pick any $T> 1 + T_0 + T_*$.

From Proposition \ref{prop:phiexistence}, for this $T$, for each $k=1, \cdots, N$, there is an 
$\epsilon_k$ and a $\lambda_k$ dependent on $\ep_k$ such that, for any $\lambda \geq \lambda_k$,
\begin{equation}
\phi_k(x,t) := e^{\lambda ( \| \kappa \|_\infty + \kappa(x) - \epsilon_k ( t - \alpha_k(x) )^2 )}, \qquad
(x,t) \in \Bbar \times [-\Tbar,\Tbar]
\label{eq:phikdef}
\end{equation}
has the properties (a), (b), (c) of Proposition \ref{prop:phiexistence}, for the operator $\L$, on the region $Q$. 
Choose a $\lambda$ larger than $\max(\lambda_0, \lambda_1, \cdots, \lambda_N)$, then the $\phi_k$ defined 
by \eqref{eq:phikdef} has properties the properties (a), (b), (c) of Proposition \ref{prop:phiexistence}, for the 
operator $\L$, on the region $Q$. Observe that
\begin{align}
\mu(x) := \phi_k(x, \alpha_k(x) ) = e^{\lambda ( \| \kappa \|_\infty + \kappa(x))}, \qquad x \in \Bbar;
\label{eq:mudef}
\end{align}
is independent of $k$ and strongly pseudoconvex w.r.t $\calE$ on $\Bbar$. 

By hypothesis $U_k=U'_k$ on $\Sigma$. Since $T>T_0$, Proposition \ref{prop:eictransfer} implies that 
$(g',\omega_k, T_0)$ has the EIC and $\alpha_k = \alpha'_k$ on $\pa B$. So the sets 
$\Sigma_{\pm,k}$ and $\Sigma'_{\pm,k}$, associated with $\omega_k$,  for $g$ and $g'$, are identical. 
Further, as argued in the proof of Theorem \ref{thm:cunique}, we have 
\[
u_k = u'_k, ~ \pa_\nu u_k = \pa_\nu u'_k, ~~~ \text{on } \Sigma_{+,k};
\qquad
v_k = v'_k, ~ \pa_\nu v_k = \pa_\nu v'_k, ~~~ \text{on } \Sigma_{-,k}.
\]
Hence
\begin{equation}
w_k = w_k', ~~ \pa_\nu w_k = \pa_\nu w'_k, \qquad \text{on } \pa B \times [-T,T].
\label{eq:wwpbdry}
\end{equation}
We also observe that, because of the EIC,
\begin{equation}
-1 \leq \alpha_k(x), \alpha_k'(x) < T_0, \qquad x \in \Bbar
\label{eq:alphabound}
\end{equation}

To prove the theorem, one could consider examining the differences $u_k-u_k'$ and $v_k-v'_k$. However,
$u_k$ is uniquely defined only on the 
region $t \geq \alpha_k(x)$, while $u'_k$ is uniquely defined only
on the region $t \geq \alpha_k'(x)$. While 
$\alpha_k=\alpha_k'$ on $\pa B$, they may differ in $B$.
So $u_k-u_k'$ can be defined only on the region $t \geq \max(\alpha_k(x), \alpha'_k(x))$, 
which is not conducive to obtaining the desired estimates. A similar issue arises if we attempt to work with 
$v_k-v_k'$.
In \cite{romanov2002}, Romanov had the important idea of generating new functions from $u_k',v_k'$, so that
the domains of the new functions were contained in the domains of $u_k,v_k$.  Romanov defined the new function
\begin{equation}
(x,t) \to u_k'(x, t - \alpha_k(x) + \alpha'_k(x));
\label{eq:upshift}
\end{equation}
since $u_k'$ is uniquely defined on the subset of $\Bbar \times \R$ where $t \geq \alpha'_k(x)$, the new function is 
uniquely defined on the region
\[
\{ (x,t) : x \in \Bbar, ~ \alpha_k'(x) \leq t - \alpha_k(x) + \alpha_k'(x) \leq T \},
\]
so on the region
\[
\{ (x,t) : x \in \Bbar, \alpha_k(x) \leq t \leq T + \alpha_k(x) - \alpha'_k(x) \}.
\]
Now, for $x \in \Bbar$,
\[
T + \alpha_k(x) - \alpha'_k(x) \geq  T -1 - T_0 > T_*,
\]
so the function \eqref{eq:upshift} is defined at least on the subset of $\Bbar \times [-T_*, T_*]$ where
$t \geq \alpha_k(x)$. 
Similarly, the function
\[
(x,t) \to v_k'(x, t - \alpha_k(x) + \alpha'_k(x))
\]
is defined at least on the region on the region $x \in \Bbar$, $ - T_* \leq t \leq \alpha_k(x)$.

From now onwards we work in the region $\Bbar \times [-T_*, T_*]$. {\bf Rather than choosing new labels, we continue
to use the definition of $Q, \Sigma, \Sigma_{\pm}, \cdots$ but with $T_*$ replacing $T$.}

For each $k=1, \cdots, N$, define the functions
\begin{alignat*}{3}
& \Abar(x) &&= A(x) - A'(x), &&\qquad x \in \Bbar,
\\
& \alphabar_k(x) &&= \alpha_k(x) - \alpha_k'(x), &&\qquad x \in \Bbar,
\\
& \ubar_k(x,t) &&= u_k(x,t) - u'_k(x, t - \alphabar_k(x)),  &&\qquad (x,t) \in Q_{+,k}
\\
& \wbar_k(x,t) &&= w_k(x,t) - w_k'(x,t - \alphabar_k(x)), && \qquad (x,t) \in Q_{+,k}.
\end{alignat*}
Since $\alpha_k = \alpha'_k$ on $\pa B$, we have $\alphabar_k=0$ on $\pa B$, hence \eqref{eq:wwpbdry} implies
\begin{equation}
\wbar_k =0, \qquad \pa_\nu \wbar_k =0, \qquad \text{on } \Sigma.
\end{equation}
So Proposition \ref{prop:aux} applied to $\wbar_k$, for the operator $\L$, over the region $Q$, 
using the weight $\vphi_k$ gives
\begin{equation}
\|\wbar_k\|_{1,\Gamma}^2 \cleq \|\L \wbar_k\|_{1,\sigma,Q}^2 + \sigma^2 e^{2 \sigma \vphi_H} \|
\L \wbar_k\|_{0,Q}^2,
\label{eq:carlA}
\end{equation}
for large enough $\sigma$, $k=1, \cdots, N$.

%%%
For the RHS of \eqref{eq:carlA}, we compute $\L \wbar_k$ on $Q$. Noting that $\L' w_k'=0$ on $\Bbar \times (-\infty, T]$, from the calculations in 
subsection \ref{subsec:standard} we see that
\begin{align}
\L (w_k'(x, t - \alphabar_k(x))) & =  F(x,t - \alphabar_k(x)), \qquad (x,t) \in Q
\label{eq:Lwpta}
\end{align}
where
\[
F = (\L - \L')w_k' + 2 (\nabla \alphabar_k)^T A \, \nabla ( \pa_t w_k') - 
(\nabla \alphabar_k)^T A (\nabla \alphabar_k) \pa_t^2 w_k'+ 
((\calE -c) \alphabar_k )\pa_t w_k'.
\]
Hence, using $\L w_k=0$ on $\Bbar \times (-\infty, T]$, we have
\begin{align*}
(\L \wbar_k)(x,t) & = - F(x, t - \alphabar_k(x)), \qquad (x,t) \in Q.
\end{align*}
Noting that $(\L-\L')w_k' = -\abar_{ij} \pa_i \pa_j w_k'$, we have
\begin{align}
|\L \wbar_k| \cleq |\Abar|_\infty + |\nabla \alphabar_k| + |(\calE-c) \alphabar_k|, \qquad \text{on } Q.
\label{eq:rightcarl}
\end{align}

For the LHS of \eqref{eq:carlA}, we examine $\wbar_k$ on $\Gamma_k$. 
Define $\fbar_{0,k} = f_{0,k} - f_{0,k}'$. For $x \in \Bbar$, we have
\begin{align}
\wbar_k(x,\alpha_k(x)) & =  w_k(x,\alpha_k(x)) - w'_k(x, \alpha_k(x) - \alphabar_k(x))
= u_k(x,\alpha_k(x)) - u_k'(x,\alpha_k'(x))
\nn
\\
& = f_{0,k}(x) - f_{0,k}'(x) = \fbar_{0,k}(x).
\label{eq:uffp0}
\end{align}
From  \eqref{eq:f0Ade} and \eqref{eq:transportA} we have 
\begin{align*}
2 (\nabla \alpha_k)^T A (\nabla f_{0,k}) + ((\calE-c) \alpha_k) f_{0,k} & =0, \qquad \text{on } \Bbar,
\\
2 (\nabla \alpha_k')^T A' (\nabla f'_{0,k}) + ((\calE'-c) \alpha_k') f_{0,k}' &=0, \qquad \text{on } \Bbar.
\end{align*}
Using these equations and some algebraic manipulations (see subsection \ref{subsec:standard}) one obtains
\begin{align}
\left ( 2 (\nabla \alpha_k)^T A \,\nabla  + ((\calE-c) \alpha_k) \right ) & ( \wbar_k(x, \alpha_k(x)) )
= \left ( 2 (\nabla \alpha_k)^T A \,\nabla  + ((\calE-c) \alpha_k) \right )( \fbar_{0,k})
\nn
\\
& = - \calP_k \alphabar_k - 2 (\nabla \alpha_k')^T \Abar (\nabla f_{0,k}') - (\abar_{ij} \pa_i \pa_j \alpha_k') f_{0,k}',
\label{eq:firstalphabar}
\end{align}
where $\calP_k$ is the elliptic operator
\[
\calP_k :=  f_{0,k}' \,(\calE -c) + 2 (\nabla f_{0,k}')^T A \, \nabla.
\]
Note that $f_{0,k}'(x)>0$ on $\Bbar$ because, from \eqref{eq:transport} and \eqref{eq:f0Ade},
$f'_{0,k}$ is the solution of a homogeneous linear first order ODE (on a geodesic) with a positive initial condition. 
Hence, using \eqref{eq:firstalphabar} and \eqref{eq:mudef}, we have
\begin{align}
\| \wbar_k\|_{1,\Gamma_k,\sigma}^2 & \cgeq \int_B e^{2 \sigma \vphi_k(x, \alpha_k(x))}
\left ( | (\calP_k \alphabar_k)(x)|^2 - |\Abar(x)|_\infty^2 \right ) \, dx
\nn
\\
& = \int_B e^{2 \sigma \mu(x))}
\left ( | (\calP_k \alphabar_k)(x)|^2 - |\Abar(x)|_\infty^2 \right ) \, dx,
\label{eq:GammaP}
\end{align}
with the constant independent of $\sigma$.

Therefore using \eqref{eq:rightcarl} and \eqref{eq:GammaP} in \eqref{eq:carlA}, we obtain
\[
\| \calP_k \alphabar_k\|_{0,\sigma, \Gamma} - \| |\Abar|_\infty \|_{0,\sigma, \Gamma}
\cleq \| [ \nabla \alphabar_k, (\calE - c)\alphabar_k] \|_{0,\sigma,Q} + e^{2 \sigma \vphi_H} 
\| [ \nabla \alphabar_k, (\calE - c)\alphabar_k]\|_{0,Q},
\]
for large enough $\sigma$. Since $\calP_k$ is elliptic and has the same\footnote{up to a multiplication by a positive function} principal symbol as $\calE$, $\mu$ is strongly pseudoconvex w.r.t $\calP_k$ on $\Bbar$. Hence, applying Proposition \ref{prop:conseq}, we obtain
\begin{align}
 \sigma \int_B e^{2 \sigma \mu} (|\nabla \alphabar_k|^2 + \sigma^2 |\alphabar_k|^2)
\cleq \int_B e^{2 \sigma \mu} |\Abar|_\infty^2,
\qquad k=1, \cdots, N,
\label{eq:kest}
\end{align}
for large enough $\sigma$. It remains to convert this estimate of
$\alphabar_k$ to an estimate of $\Abar$.

For $n \times n$ real matrices $M_1, M_2$, let $M_1 \odot M_2$ denote
the sum of the entries of the Hadamard product $M_1 \odot M_2$; $M_1 \odot M_2$ 
is the usual inner product on the vector space
of real matrices. If $v$ is a column vector in $\R^n$ and $M$ a symmetric matrix with real entries then
$(vv^T) \odot M = v^TMv$.

By hypothesis, for each $x \in \Bbar$, $(\nabla \alpha_k(x)) (\nabla \alpha_k(x))^T $, $k=1, \cdots, N$ span the space 
of symmetric real matrices. Hence, for every non-zero symmetric real matrix $M$, we have
\[
\sum_{k=1}^N | \nabla \alpha_k(x)^T M \, \nabla \alpha_k(x)| 
= \sum_{k=1}^N | \nabla \alpha_k(x) (\nabla \alpha_k(x) )^T \odot M| > 0.
\]
Hence, using continuity and the compactness of $\Bbar \times \{M : M \odot M =1 \}$, we have
\[
\sum_{k=1}^N | \nabla \alpha_k(x)^T M \, \nabla \alpha_k(x)|  \cgeq  \sqrt{M \odot M} \cgeq |M|_\infty,
\qquad x \in \Bbar,
\]
with the {\bf constant independent of $M$}. So
\begin{align}
\sum_{k=1} | (\nabla \alpha_k)^T \Abar \, \nabla \alpha_k|  \cgeq  |\Abar|_\infty,
\qquad \text{on }  \Bbar,
\label{eq:spanAalpha}
\end{align}
with the constant independent of $x \in \Bbar$.
Now $(\nabla \alpha_k)^T A (\nabla \alpha_k)=1$ and
$(\nabla \alpha'_k)^T A' (\nabla \alpha'_k)=1$, so
\begin{align*}
(\nabla \alpha_k)^T \Abar (\nabla \alpha_k)
& = (\nabla \alpha_k)^T A (\nabla \alpha_k) - (\nabla \alpha_k)^T A' (\nabla \alpha_k)
\\
& = 1 - (\nabla \alphabar_k)^T A' (\nabla \alpha_k)
- (\nabla \alpha'_k)^T A' (\nabla \alpha_k)
\\
& = 1 - (\nabla \alphabar_k)^T A' (\nabla \alpha_k)
- (\nabla \alpha'_k)^T A' (\nabla \alphabar_k) - (\nabla \alpha'_k)^T A' (\nabla \alpha'_k)
\\
&=  - (\nabla \alphabar_k)^T A' (\nabla \alpha_k)
- (\nabla \alpha'_k)^T A' (\nabla \alphabar_k).
\end{align*}
Hence, using \eqref{eq:spanAalpha}, we have
\begin{align}
|\Abar|_\infty & \cleq \sum_{k=1}^N   |\nabla \alphabar_k | + |\alphabar_k|,
\qquad \text{on } \Bbar,
\label{eq:Aalphaest}
\end{align}
with the constant independent of $x \in \Bbar$.
Using \eqref{eq:Aalphaest} in \eqref{eq:kest}, for large enough $\sigma$,
we obtain
\[
 \sigma \int_B e^{2 \sigma \mu}  | \Abar|_\infty^2 \leq \int_B e^{2 \sigma \mu}  | \Abar|_\infty^2,
\]
therefore $\Abar=0$ on $\Bbar$, and we have proved the theorem.

%%%%%%%%%%%%%%%%%%%
%%%%%%%%%%%%%%%%%
\appendix

\section{Appendix}

%%%%%%%%
%%%%%%%%%

\subsection{Proof of Proposition \ref{prop:uforward}}\label{subsec:uforward}

We prove Proposition \ref{prop:uforward} for $U_c$ and the proof for $U_{A,\omega}$ is
similar so we do not give its proofs. Note that $\L$ is {\em normally hyperbolic} as defined in section 1.5 of 
\cite{bgp07} as 
its principal symbol is associated with the Lorentzian metric $-dt^2 + a_{ij}(x) dx_i dx_j$ on $\R^n \times \R$ 
with $\pa_t$ as the global future directed timelike vector field. Further, $\L$ is {\em globally hyperbolic}
(see Section 2 and Proposition 2.3 of \cite{ors26a} for several equivalent definitions) because $t=c$ is a smooth 
spacelike Cauchy surface for this Lorentzian metric. Hence, from \cite[Corollary 3.4.3]{bgp07}, $\L$ has advanced and retarded Green's operators $G_+, G_-$ as defined in \cite[Definition 3.4.1]{bgp07}. Therefore 
\cite[Theorem 3.8]{bar15} implies the existence of the linear extensions $\Gbar_+, \Gbar_-$, hence
\cite[Lemma 4.1]{bar15} holds for $\L$.
To keep the notation simple, we do not display the dependence on $c,g,\omega$, so $U_c, \Lambda_{g,\omega}$ and 
$\alpha_{g,\omega}$ are written as $U, \Lambda$ and $\alpha$.

We prove (a) with \eqref{eq:Ucic} replaced by 
\begin{equation}
U(x,t) = \delta(t-x \cdot \omega), \qquad \text{on } \R^n \times (-\infty, -4).
\label{eq:Ucic4}
\end{equation}
Later we show why this is equivalent to the original statement. 

Let $\chi(t)$ be a smooth function on $\R$ with
\[
\chi(t) = \begin{cases} 1 & t < -3 \\ 0 & t \geq - 2 \end{cases}.
\]
A distribution $U$ on $\R^n \times \R$ is a solution of \eqref{eq:Ucde}, \eqref{eq:Ucic4} iff 
$W = U - \chi(t) \delta(t-x \cdot \omega)$ is a solution of
\begin{subequations}
\begin{align}
\L W = f, & \qquad \text{on } \R^n \times \R,
\label{eq:wcde}
\\
W =0, & \qquad \text{on } \R^n \times (-\infty, -4),
\label{eq:wcic}
\end{align}
\end{subequations}
where
\[
f = - \L (\chi(t) \delta(t- x \cdot \omega) ).
\]
Now $\chi(t) \delta(t - x \cdot \omega)$ is supported in the region $ \{ (x,t) \in \R^n \times \R : x \cdot \omega \leq t \leq -2 \}$
and this region does not intersect $\Bbar \times \R$. Hence $\L = \Box$ on this region so, noting $\Box \delta(t- x \cdot \omega) =0$ on $\R^n \times \R$, we have
\[
f = - [\Box, \chi(t)] \delta (t- x \cdot \omega),
\]
so $f$ is supported in
\[
M := \{ (x,t) \in \R^n \times \R :  \max(-3, x \cdot \omega) \leq t \leq -2\}.
\]
Hence $f$ is a past compact (see \cite{bgp07} for the definition) distribution on $\R^n \times \R$. So, applying \cite[Lemma 4.1]{bar15}, 
\eqref{eq:wcde}has a solution $W$ which is a distribution on $\R^n \times \R$ with
\[
\tsupp (W) \subset J^+( M) \subset \R^n \times [-3, \infty),
\]
so \eqref{eq:wcic} holds. Further, because of \eqref{eq:wcic}, any solution of \eqref{eq:wcde}, \eqref{eq:wcic} will be 
past compact hence, by \cite[Corollary 4.2]{bar15}, \eqref{eq:wcde}, \eqref{eq:wcic} has at most one distributional solution. Hence \eqref{eq:Ucde}, \eqref{eq:Ucic4} has a unique distributional solution.

We now prove the claim about the support of $U$. One may quickly check that the support of 
$ \chi(t) \delta (t- x \cdot \omega) $ is a subset of
\[
 Q_\alpha:= \{ (x,t) \in \R^n \times \R : \alpha(x) \leq t \}.
\]
So $\tsupp(U) = \tsupp (W + \chi(t) \delta(t- x \cdot \omega)) \subset  J_+(M) \cup Q_\alpha$. Now 
\[
M \subset N := \{ (x,t) \in \R^n \times \R : x \cdot \omega \leq t \leq -1 \},
\]
so it is enough to show that $J_+(N) \subset Q_\alpha$. 

We first show that
\[
J_+(N) \cap \{ x \cdot \omega \leq -1 \} = Q_\alpha \cap \{ x \cdot \omega \leq -1 \}.
\]
$ \cap \{ x \cdot \omega \leq -1 \}$. Then $t_0 \geq x_0 \cdot \omega$

Suppose $(x_0,t_0) \in N$ and $(x_1, t_1) \in J_+(N)$. There is a curve $t \in [t_0, t_1] \to \gamma(t) \in \R^n$ from $
x_0$ to $x_1$ with $\|\gammadot(t)\| \leq 1$. Hence
\[
t_1 - t_0 \geq \int_{t_0}^{t_1} \|\gammadot(t)\| \, dt \geq d(x_0, x_1).
\]
If $x_1$ is also in the region $x_1 \cdot \omega \leq -1$, the shortest distance between $x_1, x_0$ will either be 
attained by a geodesic staying always in the region $x \cdot \omega \leq -1$ or by a geodesic which goes into the region $x \cdot \omega \geq -1$. In either case
\[
d(x_0, x_1) \geq | x_1 \cdot \omega - x_0 \cdot \omega| \geq x_1 \cdot \omega - x_0 \cdot \omega.
\]
Hence
\[
t_1 - x_1 \cdot \omega \geq t_0 - x_0 \cdot \omega \geq 0.
\]
If $x_1$ lies in the region $x \cdot \omega >-1$ then
\[
d(x_0, x_1) \geq \min_{a \in P_\omega} d(a, x_1) + \min_{a \in P_\omega} d(a, x_0)
= \alpha(x_1) - x_0 \cdot \omega
\]
hence $t_1 - t_0 \geq \alpha(x_1) - x_0 \cdot \omega$, so
\[
t_1 - \alpha(x_1) \geq t_0 - x_0 \cdot \omega \geq 0.
\]

To prove the existence of the solution of the original IVP \eqref{eq:Ucde}, \eqref{eq:Ucic}, we just repeat the proof of
the existence for the \eqref{eq:Ucde}, \eqref{eq:Ucic4} case but with $\chi=1$ on the region $t \leq -1$ and zero for
$t \geq 0$. Everything goes through as before. {\em The only reason we replaced \eqref{eq:Ucic} by \eqref{eq:Ucic4}
was to prove the support property.} Now any any solution of the IVP \eqref{eq:Ucde}, \eqref{eq:Ucic} is also a solution
of \eqref{eq:Ucde}, \eqref{eq:Ucic4}, so the uniqueness result proves that the solutions corresponding to the two different initial conditions are the same.

Now we determine $\tWF(U)$. Since \eqref{eq:Ucde} is a homogeneous equation, $\tWF(U)$ is contained in
the characteristic set of $\L$. Every element in the characteristic set of $\L$ lies on a null bicharacteristic, and every null bicharacteristic enters the region $\{ t<-1\}$. By the propagation of singularities theorem
(see \cite[Theorem 26.1.1]{hor85}), $\tWF(U)$ is invariant under the bicharacteristic flow associated with $\L$. 
Now $U = \delta(t- x \cdot \omega)$ for $t<-1$ so 
\begin{align*}
\tWF(U|_{t<-1}) &= \{ [ x, x \cdot \omega; - \tau \omega, \tau] : 
x \in \R^n, ~ x \cdot \omega < -1, ~ \tau \in \R, ~ \tau \neq 0\};
\\
& = \{ [a + (t+1) \omega; - \tau \omega, \tau] : a \in A, ~ t < -1, ~ \tau \in \R, ~ \tau \neq 0 \}
\\
& = \Lambda \cap \{ t<-1 \}.
\end{align*}
Hence $\tWF(U)$ is the flow out of $\tWF(U|_{t<-1})$ under the bicharacteristic flow of $\L$, which is exactly
$\Lambda$, by definition.

Finally, we show that $U \in H^{-1}_{loc}(\R^n \times \R)$. Since 
\[
U(x,t) = \delta(t-x \cdot \omega)= \pa_t ( H(t- x \cdot \omega) ), \qquad \text{for } t<-1,
\]
we see that $U |_{t<-1}\in H^{-1}_{loc}(\R^n \times (-\infty,-1))$. So, using \cite[Chapter VI, Theorem 2.1]{taylor1981} - the 
$H^s$ propagation of singularities theorem, using \cite[Chapter VI, Proposition 1.10]{taylor1981} - the regularity result 
for the elliptic directions, and imitating the argument in the previous 
paragraph, we conclude that $U \in H^{-1}_{loc}(\R^n \times \R)$.

%%%%%%%%%%%
%%%%%%%%%%%%

\subsection{Proof of Proposition \ref{prop:eictransfer}}\label{subsec:eictransfer}

By hypothesis, $U_{A,\omega} = U_{A',\omega}$ on $\pa B \times (-\infty, T_0)$ and $\L = \L' = \Box$ on
$(\R^n \setminus B) \times (-\infty, T_0)$. So, by \cite[Proposition A.1]{ors26a}, we have $U_{A,\omega} = U_{A',\omega}$
on $(\R^n \setminus \Bbar) \times (-\infty, T_0)$. Hence
\[
\tWF(U_{A,\omega})|_{(\R^n \setminus \Bbar) \times (-\infty, T_0)}
= 
\tWF(U_{A',\omega})|_{(\R^n \setminus \Bbar) \times (-\infty, T_0)}
\]
so, using Proposition \ref{prop:uforward}, for $g=A^{-1}, g'=A'^{-1}$,
we have
\[
\Lambda_{g,\omega}|_{(\R^n \setminus \Bbar) \times (-\infty, T_0)} = 
\Lambda_{g',\omega}|_{(\R^n \setminus \Bbar) \times (-\infty, T_0)}.
\]
We do not claim (yet) the above relation at points on $\pa B \times (-\infty, T_0)$ because 
$U_{A,\omega} = U_{A',\omega}$ on $\pa B \times (-\infty, T_0)$ only gives the equality of the wave front sets
of the traces, hence there is an ambiguity about the sign of the outward normal component of velocity of the geodesic 
at a boundary point. 

From the definition of $\Lambda_{g,\omega}$ and $\Lambda_{g',\omega}$ (as a subset of $T\R^{n+1}$) we see that 
for any $x \in \R^n \setminus \Bbar$, we have $x = \gamma_{p,g,\omega}(r)$ for some $p \in P_\omega$, $r<T_0$ iff 
$x=\gamma_{p',g',\omega}(r)$ for some $p' \in P_\omega$ and, for such $x$, 
$\gammadot_{p,g,\omega}(r) = \gammadot_{p',g',\omega}(r)$.  Hence
\[
D_{g,\omega} \setminus \Bbar = D_{g',\omega} \setminus \Bbar.
\]
Further, since $(g,\omega,T_0)$ satisfies the EIC, for every $x \in D_{g',\omega} \setminus \Bbar$ there is a unique
$p \in P_\omega$ and an $r<T_0$ such that $x=\gamma_{p,g',\omega}(r)$, and
\[
\{ x \in \R^n : x \cdot \omega \leq1 \} \setminus \pa B \subset D_{g', \omega}.
\]
Also, examining the singular supports of $U_{g,\omega}|_{\pa B \times \R} = U_{g',\omega}|_{\pa B \times \R}$ and
that $(g,\omega,T_0)$ satisfies the EIC, we see that $\pa B \subset D_{g',\omega}$, hence
\[
D_{g,\omega} \setminus B = D_{g',\omega} \setminus B.
\]

Now we resolve the ambiguity of the sign of the outward normal component of $\gamma_{p,g',\omega}$ at points
on $\pa B$. Any geodesic $\gamma_{p,g',\omega}$ which reaches a point $x \in \pa B$, must either have 
velocity $\omega$ at $x$ and $x \cdot \omega \leq 0$, or the velocity at $x$ must have a positive outward normal 
component. At points $x \in \pa B$ with $x \cdot \omega \leq 0$, $\gammadot_{p,g',\omega}$
cannot have a positive outward normal component because then $\gammadot_{g',p,\omega}$
would be different from $\omega$ at points near $x$
in $\R^n \setminus \Bbar$. This would contradict the claim in the previous paragraph because, by the EIC,
$\gamma_{p,g,\omega}$ has velocity $\omega$ at all points in the region $x \cdot \omega \leq 0$.
Hence $U_{A,\omega} = U_{A',\omega}$ on $\pa B \times (-\infty, T_0)$ implies
\[
\Lambda_{g,\omega}|_{\pa B \times (-\infty, T_0)} = 
\Lambda_{g',\omega}|_{\pa B\times (-\infty, T_0)},
\]
and arguing as before one can show that, for each $x \in \pa B$, there is exactly one $p \in P_\omega$ and $r<T_0$ such 
that $x = \gamma_{p,g',\omega}(r)$. This completes the proof of the claim that $(g',\omega, T_0)$ has the EIC property,
 $D_{g,\omega} = D_{g',\omega}$, and if 
\[
\gamma_{p,g,\omega}(r) = \gamma_{p',g',\omega}(r') \in D_{g,\omega} \setminus B,
\]
for some $p,p' \in P_\omega$ and $r,r'<T_0$, then $r=r'$ and 
$\gammadot_{p,g,\omega}(r) = \gammadot_{p',g',\omega}(r')$. Hence, from Proposition \ref{prop:exinj}),
\[
\alpha_{g,\omega} = \alpha_{g',\omega} \qquad \text{on } D_{g,\omega} \setminus B.
\]
We now show that $\alpha_{g,\omega} = \alpha_{g',\omega}$ on $\R^n \setminus B$.

Since $D_{g,\omega} = D_{g',\omega}$ contains the region $x \cdot \omega \leq 1$, it remains to show the equality of 
the $\alpha's$ only at points $x_0$ in the region $x \cdot \omega >1$. Define
\[
r_0 := \alpha_{g,\omega}(x_0).
\]
From \cite{ors26a}, we know there is a $p \in P_\omega$ so that
\[
\gamma_{p,g,\omega}(r_0) = x_0,
\]
and, from connectedness, there is a $q$ on $x\cdot \omega =1$ and an $r_1 < r_0$ so that
\[
\gamma_{p,g,\omega}(r_1) = q.
\]
Further, from the definition of $\alpha_{g,\omega}$, we have $r_1 < T_0$. Let $\sigma_1$ be the part of
$\gamma_{p,g,\omega}$ from $q$ to $x_0$. From the EIC property, $\sigma_1$ lies in the region 
$x \cdot \omega \geq 1$.
Hence the length of $\sigma_1$, in the metrics $g$ and $g'$ is $r_0-r_1$. 
Since $\alpha_{g,\omega}(q) = \alpha_{g',\omega}(q)$, there is a $p' \in P_\omega$ such that
\[
\gamma_{p',g',\omega}(r_1) = q.
\]
Let $\sigma_1'$ be the curve, consisting of the part of $\gamma_{p',g',\omega}$ from $p'$ to $q$. Then,
the length of $\sigma_1'$ in the $g'$ metric is $r_1+1$. Let $\sigma$ be the curve from 
$p' \in P_\omega$ to $x_0$ consisting of
$\sigma_1'$ followed by $\sigma_2$; then the length of $\sigma$ in the $g'$ metric is $(r_1+1) + (r_0 - r_1) = r_0+1$.
Hence, by definition,
\[
\alpha_{g',\omega}(x_0) \leq (r_0 + 1) -1 = \alpha_{g,\omega}(x_0).
\]
Reversing the roles of $g$ and $g'$, we get the reverse inequality, hence 
$\alpha_{g',\omega}(x_0) = \alpha_{g,\omega)}(x_0)$.

%%%%%%%%
%%%%%%%%%%

\subsection{The detailed structure of $U_c$ and $U_{A,\omega}$ under the EIC}\label{subsec:progressing}

To keep the expressions intelligible, we do not display the dependence on $g$ and $\omega$, 
so $\alpha_{g,\omega}, D_{g,\omega}$ and $\T_{g,\omega}$ are written as $\alpha, D, \T$

%%%%%%%%%%%%%%%%%%%%%%%%%
\begin{prop}[The structure of $U_c$ and $U_{A,\omega}$]\label{prop:ualpha}
Suppose $\L$ is an admissible operator, $(g=A^{-1},\omega,T_0)$ satisfies the EIC, and $m$ is a positive integer. 
\vspc
\begin{enumerate}[(a)]
\item If $U_c$ is the solution of \ref{eq:Ucde}, \ref{eq:Ucic} then
\[
U_c(x,t) = f_{-1}(x) \, \delta(t-\alpha(x)) + 
\sum_{i=0}^m f_i(x) K_i(t-\alpha(x)) + R_m(x,t), \qquad (x,t) \in D \times \R, 
\]
for some $f_i \in C^\infty(D)$, $R_m(x,t) \in H^{m+1}(\Omega)$ for every open $\Omega$ compactly contained
in $D \times \R$. The $f_i$ are the solutions of the IVP
\begin{subequations}
\begin{align}
(2 \T  + (\calE -c)\alpha))f_{-1} =0 ~~ \text{on } D, & \qquad f_{-1}(x)= 1 ~~ 
\text{when } x \cdot \omega \leq -1,
\label{eq:f1cde}
\\
(2\T + (\calE -c) \alpha) f_i = \calE f_{i-1}  ~~ \text{on } D, & \qquad f_i(x) =0 ~~ 
\text{on } x \cdot \omega \leq -1, ~~ i=0, \cdots, m.
\label{eq:ficde}
\end{align}
\end{subequations}
Further
\[
U_c(x,t) = f_{-1}(x) \, \delta(t-\alpha(x)) + u_c(x,t) \,H(t-\alpha(x)), \qquad (x,t) \in D \times \R, 
\]
with $u_c$ a smooth function on the region $D \times \R$ satisfying
\begin{subequations}
\begin{align}
\L u_c =0, & \qquad \text{for } x \in D, ~\alpha(x) \leq t,
\label{eq:ucde}
\\
(2\T + (\calE-c) \alpha)u_c = \calE f_{-1} , & \qquad \text{for  }  x \in D, ~ t=\alpha(x),
\label{eq:uccc}
\\
u_c = 0, & \qquad \text{for } x \in D, ~\alpha(x) \leq t < -1.
\label{eq:ucic}
\end{align}
\end{subequations}
\item
If $U_{A,\omega}$ is the solution of \eqref{eq:UAde}, \eqref{eq:UAic} then 
\[
U_{A,\omega}(x,t) =  \sum_{i=0}^m f_i(x) K_i(t-\alpha(x)) + R_m(x,t), \qquad (x,t) \in 
D \times \R, 
\]
for some $f_i \in C^\infty(D)$, $R_m(x,t) \in H^{m+1}(\Omega)$ for every open $\Omega$ compactly contained
in $D \times \R$.  The $f_i$ are the solutions of the IVP
\begin{subequations}
\begin{align}
(2\T + (\calE-c) \alpha) f_0 =0 ~~ \text{on } D, & \qquad f_0(x)= 1 ~~ \text{when } x \cdot \omega \leq -1,
\label{eq:f0Ade}
\\
(2\T +  (\calE-c) \alpha) f_i = \calE f_{i-1}  ~~ \text{on } D, & 
\qquad f_i(x) =0 ~~ \text{when } x \cdot \omega \leq -1, \qquad i=1, \cdots, m.
\label{eq:fiAde}
\end{align}
\end{subequations}
Further
\[
U_{A,\omega}(x,t) = u_{A,\omega}(x,t) H(t - \alpha(x) ), \qquad (x,t) \in D \times \R,
\]
with $u_{A,\omega}$ a smooth function on the region $D \times \R$ satisfying
\begin{align}
\L u_c =0,  \qquad & \text{for } x \in D, ~\alpha(x) \leq t,
\label{eq:uAde}
\\
u_{A,\omega}(x, \alpha(x)) = f_0(x), ~~ x \in D, \qquad & u_{A,\omega}(x,t) = 1 ~~\text{when~} ~\alpha \leq t < -1.
\label{eq:uAcc}
\end{align}
\end{enumerate}
\end{prop}
\vspc
\noindent
Observe that $f_{-1}(x)$ is independent of $c$ because $\T$ and $\calE -c$ are independent of $c$. 
Also, the uniqueness of $U_c, U_{A,\omega}$ implies that 
$u_c, u_{A,\omega}$ are uniquely determined on the region $\{(x,t) : x \in D, ~\alpha(x) \leq t \}$.

%%%%%
%%%%%%
\noindent
\underline{Proof of Proposition \ref{prop:ualpha}}

We prove the part of Proposition \ref{prop:ualpha} associated with the solution $U_c$ of \eqref{eq:Ucde}, 
\eqref{eq:Ucic}. The proof for the solution $U_{A,\omega}$ is almost identical and we do not give that proof. 
For the proof for the $U_c$ case, since we work with 
a fixed $A,b,c$ and $\omega$ and $g=A^{-1}$, to keep the presentation intelligible, our notation does not display 
the dependence on $A,b,c,g,\omega$. 

The distributions $K_i(t- \alpha(x))$ are defined on $\R^n \times \R$ for $i \geq -1$, but their derivatives are complicated
since $\alpha(x)$ is not smooth on $\R^n$. Now $(A,\omega,T_0)$ satisfies the EIC so, as shown in 
Section \ref{sec:intro}, $\alpha(x)$ is a smooth function on $D$. Hence, on $D \times \R$, one can compute
the derivatives of the distributions $K_i(t - \alpha(x))$ using the chain rule.

We aim to express $U$ over $D \times \R$ in the form of a finite progressing wave plus a smooth enough function. 
Fix a positive integer $m$ and define
\begin{equation}
U_m(x,t) = \sum_{k=-1}^{m} f_k(x) K_k (t- \alpha(x)), \qquad (x,t) \in \R^n \times \R.
\label{eq:Umguess}
\end{equation}
A standard calculation (see subsection \ref{subsec:standard}) shows that, on $D \times \R$, we have
\begin{align}
(\L U_m)(x,t) & = 
 ( 2 \T + (\calE-c) \alpha ) f_{-1})(x) \, \delta'(t - \alpha(x)) - (\calE f_m)(x) \, K_m(t-\alpha(x))
 \nn
 \\
& \qquad + \sum_{k=0}^{m-1}  ( (2\T + (\calE -c) \alpha) f_k - \calE f_{k-1}) (x) \, K_{k-1}(t - \alpha(x) ).
\label{eq:LUmalpha}
\end{align}
We choose the $f_i(x)$ to be solutions of the IVPs below. 
\begin{subequations}
\begin{align}
(2\T + (\calE -c) \alpha) f_{-1} =0 ~~ \text{on } D, & \qquad f_{-1}(x ) = 1~~ \text{on } x \cdot \omega \leq -1,
\label{eq:bminus1de}
\\
(2\T + (\calE -c) \alpha) f_k = \calE f_{k-1}  ~~ \text{on } D, & \qquad f_k(x) =0 ~~ \text{on } x \cdot \omega \leq -1, \qquad i=0, \cdots, m.
\label{eq:bide}
\end{align}
\end{subequations}
These are first order linear ODEs on the geodesics $ r \to \gamma_p(r)$ because of \eqref{eq:transport},
hence have unique smooth solutions on $D$. With these $f_i$, we have
\begin{align*}
(\L U_m)(x,t) = - (\calE f_m)(x) \, K_m(t-\alpha(x)), & \qquad \text{on } D \times \R,
\\
U_m(x,t) = \delta(t-x \cdot \omega), & \qquad \text{for } t < -1.
\end{align*}
Choose an $\ep>0$ small enough so that $T_1 < T_0 - 2 \ep$ and choose a $\psi \in C^\infty(\R)$ with
\[
\psi(t) = \begin{cases} 1, & t \leq T_0 - 2 \ep \\ 0, & t \geq T_0 - \ep \end{cases}.
\]
Hence 
\[
\psi(t) K_i( t- \alpha(x) ) = \begin{cases} K_i( t - \alpha(x)), & \text{if } \alpha(x) \leq T_0- 2 \ep \\ 0, &
\text{if } \alpha(x) \geq T_0 - \ep \end{cases},
\]
so $\psi(t) U_m(x,t)$ is zero on a neighborhood of $(\R^n \setminus D) \times \R$ and
\[
\L  ( \psi(t) U_m(x,t) ) = - \psi(t) \,  ( \calE f_m )\, K_m(t-\alpha(x)) + 
[\pa_t^2, \psi(t)] U_m(x,t),
\qquad \text{on } \R^n \times \R.
\]
Define
\[
r_m(x,t) := U(x,t) - \psi(t) U_m(x,t), \qquad (x,t) \in \R^n \times \R.
\]
Then $r_m(x,t) = r_m'(x,t) + r''_m(x,t)$ where $r_m', r_m''$ are the solutions of the IVP problems
\begin{subequations}
\begin{align}
\L r_m' (x,t)  = - \psi(t) \,  ( \calE f_m ) \, K_m(t-\alpha(x)),
& \qquad \text{on } \R^n \times \R,
\label{eq:rmpde}
\\
r_m'(x,t) =0, & \qquad \text{for } t<-1,
\label{eq:rmpic}
\end{align}
\end{subequations}
and
\begin{subequations}
\begin{align}
\L r_m'' (x,t)  = [\pa_t^2, \psi(t)] U_m(x,t),
& \qquad \text{on } \R^n \times \R,
\label{eq:rmppde}
\\
r_m''(x,t) =0, & \qquad \text{for } t<-1.
\label{eq:rmppic}
\end{align}
\end{subequations}

Now $\psi(t) K_m(t-\alpha(x)) \in H^m_{loc}(\R^n \times \R)$ and the RHS of \eqref{eq:rmpde} is
zero for $t < -1$ because $f_m(x)$ is zero
on the region $\alpha(x) \leq -1$. Hence, from the well-posedness theory, we know that 
$r_m' \in H^{m+1}_{loc}(\R^n \times \R)$. 

Since $\psi'(t)$ is supported in $[T_0-2\ep, T_0 - \ep]$ and $U_m(x,t)$ is a conormal distribution associated with 
the surface 
\[
S := \{(x,\alpha(x)): x \in \R^n \},
\]
the wave front set of the RHS of \eqref{eq:rmppde} is a subset of
\[
N(S) \cap \pi^{-1} (\R^n \times [T_0-2\ep, T_0 - \ep] );
\]
here $N(S)$ is the normal bundle of $S$ and $\pi : T^*(\R^n \times \R) \to \R^n \times \R$ is the usual projection. 
Note that this set uses only the part of $S$ associated with $x \in D$ because we require 
$T_0-\ep \leq \alpha(x) \leq T_0 - 2 \ep$. 
If we identify $T^*(\R^n \times \R)$ with $T(\R^n \times \R)$ through $g$, then the intersection of the wave front set 
of the RHS of \ref{eq:rmppde} with the sphere bundle (of radius 2) is a subset of
\[
\Upsilon := \{ (x, \alpha(x); \nabla_g \alpha(x), 1) : x \in K \},
\]
where
\[
K := \{ x \in \R^n : T_0 - 2 \ep \leq \alpha(x) \leq T_0 - \ep \} \subset D.
\]
Now the normalized null bicharacteristics of $\L$ may be identified with the curves $t \to (\gamma(t), t; \gammadot(t), 1)$
where $\gamma(t)$ is a geodesic of $(\R^n, g=A^{-1})$. Since a geodesic is uniquely determined by a point on it and 
its velocity there, from Proposition \ref{prop:exinj}, a null bicharacteristic of $\L$
intersects $\Upsilon$ only if the geodesic
$\gamma$ is a $\gamma_{g,p,\omega}$ for some $p \in P_\omega$. Further, all these null bicharacteristics
intersect the region $t<-1$ where
$r_m''$ is zero; also $\alpha(x)$ increases along these geodesics. Hence, by the propagation of 
singularities theorem (see \cite[Chapter VI, Theorem 2.1]{taylor1981}), $r_m''$ is smooth on the region 
$D_\ep \times \R$ where
\[
D_\ep := \{ x \in \R^n : \alpha(x) < T_0 - 2 \ep \}.
\]
Therefore $r_m = r_m' + r_m'' \in H^{m+1}_{loc}(D_\ep \times \R)$, so 
$U(x,t) - \psi(t) U_m(x,t) \in H^{m+1}_{loc}(D_\ep \times \R)$.

Now, on $D_\ep \times \R$,
\[
(U - U_m)(x,t) = U(x,t) - \psi(t) U_m(x,t) + (\psi(t) -1) U_m(x,t) = r_m(x,t) + (\psi -1) U_m(x,t).
\]
Since $\psi(t)-1=0$ for $t \leq T_0 - 2 \ep$ and the intersection of the singular support of $U_m(x,t)$ with 
$D_\ep \times \R$ is a subset of $\R^n \times (-\infty, T_0 - 2 \ep)$, we see that $(\psi -1) U_m(x,t)$ is smooth
on $D_\ep \times \R$, hence $U - U_m \in H^{m+1}_{loc}(D_\ep \times \R)$. Since $\ep>0$ can be made
arbitrarily small, we see that  $U-U_m \in H^{m+1}(\Omega)$ for every open $\Omega$ compactly contained in 
$D \times \R$.

Define $R_m := U - U_m$ on $D \times \R$; then 
\[
U = U_m + R_m, \qquad \text{on } D \times \R,
\]
$R_m \in H^{m+1}(\Omega)$ for every open $\Omega$ compactly contained in $D \times \R$, and 
the support of $R_m$ is a subset of
the region $t \geq \alpha(x)$ because $U$ and $U_m$ are supported in this region. 
Also, for future use, we note that
\begin{equation}
\L R_m = \L U - \L U_m = (\calE f_m) K_m(t - \alpha(x)), \qquad \text{on } D \times \R.
\label{eq:LRm}
\end{equation}

For any $m> (n+1)/2 -1$, $R_m \in C^{m+1-(n+1)/2}(D \times \R)$. Define
\[
u(x,t) := \sum_{i=0}^m f_i(x) \frac{(t-\alpha(x))^i}{i!} + R_m(x,t), \qquad (x,t) \in D \times \R;
\]
then $u \in C^{m+1-(n+1)/2}(D \times \R)$. Further,
\[
U(x,t) = f_{-1}(x) \delta(t-\alpha(x)) + u(x,t) H(t-\alpha(x)), \qquad \text{for } x \in D, ~ t \geq \alpha(x),
\]
hence $u$ is uniquely determined on the region $\{(x,t): x \in D, ~ t \geq \alpha(x) \}$. Since this is true for every 
positive integer $m$, $u$ is a smooth function on 
the region $\{(x,t): x \in D, ~ t \geq \alpha(x) \}$ and we can choose a smooth extension of it to $D \times \R$.

Now $\L U=0$ on $\R^n \times \R$ and, on $D \times \R$, $\L (f_{-1} \delta(t-\alpha) )$ is supported on $t=\alpha(x)$. 
Further $u(x,t) H(t-\alpha(x)) = u(x,t)$ on $\{(x,t): x \in D, ~ t > \alpha(x) \}$. Hence
$\L u =0$ on $\{(x,t): x \in D, ~ t >\alpha(x) \}$ so continuity implies $\L u=0$ on the region
$\{(x,t): x \in D, ~ t \geq \alpha(x) \}$. Finally, by definition,
\[
u(x, \alpha(x)) = f_0(x), \qquad x \in D,
\]
so \eqref{eq:uccc} follows from \eqref{eq:ficde} for $i=0$.
%%%%%%%

%%%%%%%%%%%%%%%
%%%%%%%%%%%%%%%%%%

\subsection{Proof of Proposition \ref{prop:aux}}\label{subsec:aux}

%%%%%%%%%

The proof of Proposition \ref{prop:aux} uses the following energy estimates near $t=T$ and near $\Gamma$.
These are derived in the usual manner using multipliers. The terms in the integrals on $\Sigma_+$ are not optimal 
but the estimates are good enough for what we need.
\begin{lemma}\label{lemma:energy}
Suppose $\L$ is an admissible operator on $\R^n \times \R$, $(g=A^{-1},\omega,T_0)$ satisfies the EIC 
and $T>T_0$. For all $f \in C^2(Q_+)$ and all $\sigma>0$, we have the following estimates:
\begin{subequations}
\begin{align}
\int_{H_+} |\nabla_{x,t} f|^2 +  |f|^2 \, dS
& \cleq \int_{Q_+} |\L f|^2   + \int_\Gamma |\nabla_\Gamma f|^2 +  |f|^2 
+  \int_{\Sigma_+} |\nabla_{x,t}f|^2 + |f|^2
\label{eq:fHTest}
\\
\int_\Gamma |\nabla_\Gamma f|^2 + \sigma^2 |f|^2 
& \cleq \int_{Q_+} |\L f| \, |f_t|  + \sigma \int_{Q_+} |\nabla_{x,t} f|^2  + \sigma^2 |f|^2 
 + \int_{\Sigma_+} |\nabla_{x,t}f|^2 + \sigma^2 |f|^2,
\label{eq:fGammaest}
\end{align}
\end{subequations}
with the constants independent of $f$ and $\sigma$.
\end{lemma}
\vspc
\noindent
We postpone the proof of Lemma \ref{lemma:energy} to the end of this section.
%%%%%%%%%%%%

%%%%%%%%%
\begin{proof}[Proof of Proposition \ref{prop:aux}]
~

Since $\vphi$ is strongly pseudo-convex w.r.t $\L$, from \cite[Theorem A.7]{rs20a}, there is a $\sigma_0>1$ such that for
all $\sigma \geq \sigma_0$ and all $w\in C^2(Q)$, we have
\begin{align}
\sigma \|w\|^2_{1,Q,\sigma} \cleq \|\L w\|^2_{0,Q,\sigma} + \sigma ( \|w\|^2_{1, \Sigma, \sigma} 
+ \|\pa_\nu w\|^2_{0,\Sigma, \sigma})
+ \sigma ( \|w\|^2_{1,H_{\pm}, \sigma} + \|w_t\|^2_{0, H_{\pm}, \sigma}),
\label{eq:carleman}
\end{align}
where the constant is independent of $w$ and $\sigma$. 

\noindent
\underline{An upper bound on $\|w\|_{1,\sigma,\Gamma}$}

We intend using \eqref{eq:fGammaest} from Lemma \ref{lemma:energy} with $f = e^{\sigma \vphi} w$ on $Q_+$. We take
$\sigma \geq \sigma_0$.
A simple calculation shows that
\begin{align*}
|\L f| & \cleq e^{\sigma \vphi} |\L w| + e^{\sigma \vphi} ( \sigma^2 |w| + \sigma |\nabla_{x,t}w| ),
\end{align*}
with the constant independent of $w,\sigma$. 
Since $|f_t| \cleq e^{\sigma \vphi} (|w_t| + \sigma |w|)$ and $\sigma \geq \sigma_0>0$, we have
\begin{align*}
|\L f | \, |f_t| & \cleq e^{2 \sigma \vphi} \, |\L w| (|w_t| + \sigma |w|) + \sigma e^{2 \sigma \vphi} ( |\nabla_{x,t} w| + \sigma |w|)
(|w_t| + \sigma |w|)
\\
& \cleq e^{2 \sigma \phi} |\L w|^2 + \sigma e^{2 \sigma \vphi} ( |\nabla_{x,t} w|^2 + \sigma^2 |w|^2),
\end{align*}
giving us
\begin{align}
\int_{Q_+} |\L f | \, |f_t| \cleq \| \L  w\|_{0,Q_+,\sigma}^2 + \sigma \| w \|_{1,Q_+, \sigma}^2.
\label{eq:Lfest}
\end{align}
Now $w = e^{-\sigma \vphi} f$, so $|w| = e^{-\sigma \vphi} |f|$ and
$|\nabla_\Gamma w| \cleq e^{- \sigma \vphi} (|\nabla_\Gamma f| + \sigma |f|)$
so
\[
e^{2 \sigma \vphi} ( |\nabla_\Gamma w|^2 + \sigma^2 |w|^2) \cleq |\nabla_\Gamma f|^2 + \sigma^2 |f|^2;
\]
further $ | \nabla_{x,t} f|^2 + \sigma^2 |f|^2 \cleq e^{2 \sigma \phi} ( |\nabla_{x,t} w|^2 + \sigma^2 |w|^2)$.
Using these and \eqref{eq:Lfest} in \eqref{eq:fGammaest}, we obtain
\begin{align*}
\|w\|_{1,\Gamma,\sigma}^2 \cleq \|\L w\|_{0,Q_+,\sigma}^2 + \sigma \|w\|_{Q_+,1,\sigma}^2
+ \|w\|_{1,\Sigma_+,\sigma}^2 + \|\pa_\nu w\|_{0,\Sigma_+,\sigma}^2.
\end{align*}
Combining this with \eqref{eq:carleman}, for $\sigma \geq \sigma_0$, we obtain
\begin{align}
\|w\|_{1,\Gamma,\sigma}^2 
& \cleq  \|\L w\|_{0,Q,\sigma}^2 + \sigma ( \|w\|^2_{1,H_{\pm}, \sigma} + \|w_t\|^2_{0, H_{\pm}, \sigma})
 + \sigma(  \|w\|_{1,\Sigma,\sigma}^2 + \|\pa_\nu w\|_{0,\Sigma,\sigma}^2).
\label{eq:Gcarleman}
\end{align}

\noindent
\underline{Removing the $H_{\pm}$ integrals from \eqref{eq:Gcarleman}}

From \eqref{eq:fHTest} in Lemma \ref{lemma:energy} applied to $w$ (instead of $f$) and observing that
a similar estimate holds for $Q_-$, we have
\begin{align*}
\|w\|_{1, H_{\pm}}^2 + \|w_t\|_{0,H_{\pm}}^2 \cleq \|\L w\|^2_{0,Q}  + \|w\|_{1,\Gamma}^2
+ ( \|w\|_{1,\Sigma}^2 + \|\pa_\nu w\|_{0,\Sigma}^2).
\end{align*}
Hence 
\begin{align}
\|w\|_{1, H_{\pm}, \sigma}^2 + \|w_t\|_{0,H_{\pm}, \sigma}^2 
\cleq \sigma^2 e^{2 \sigma \vphi_H} \left ( \|\L w\|^2_{0,Q}  
+   \|w\|_{1,\Gamma}^2
+ \|w\|_{1,\Sigma}^2 + \|\pa_\nu w\|_{0,\Sigma}^2 \right ).
\label{eq:wTtemp}
\end{align}
Since
\[
2 \delta :=  \min_\Gamma \vphi - \max_{H_{\pm }} \vphi
\]
is positive, we have $e^{2 \sigma \vphi_H} \leq e^{-2 \sigma \delta} \min_\Gamma e^{2 \sigma \vphi}$
and $\sigma^2 \leq e^{\delta \sigma}$ for $\sigma$ large enough, hence
\[
\sigma^2 e^{2 \sigma \vphi_H}  \|w\|_{1,\Gamma}^2 \cleq  e^{-\sigma \delta}
\|w\|^2_{1,\sigma, \Gamma}.
\]
Using this in \eqref{eq:wTtemp}, we see that for $\sigma$ large enough, we have
\[
\|w\|_{1, H_{\pm }, \sigma}^2 + \|w_t\|_{0,H_{\pm }, \sigma}^2 
\cleq  e^{-\sigma \delta} \|w\|_{1,\Gamma, \sigma} +
\sigma^2 e^{2 \sigma \vphi_H} \left ( \|\L w\|^2_{0,Q}  
+ \|w\|_{1,\Sigma}^2 + \|\pa_\nu w\|_{0,\Sigma}^2 \right ).
\]
Using this in \eqref{eq:Gcarleman}, for $\sigma$ large enough, we have
\begin{align*}
\|w\|_{1,\Gamma,\sigma}^2 
& \cleq  \|\L w\|_{0,Q,\sigma}^2 + \sigma^2 e^{2 \sigma \vphi_H} \|\L w\|_{0,Q}
+ e^{ k \sigma } (  \|w\|_{1,\Sigma}^2 + \|\pa_\nu w\|_{0,\Sigma}^2),
\end{align*}
for some $k>0$ independent of $w$ and $\sigma$, with the inequality constant also independent of $w$ and $\sigma$.
\end{proof}

%%%%%%%%%%%%%%%%%
%%%%%%%%%

\begin{proof}[Proof of Lemma \ref{lemma:energy}] ~

For convenience, we write $\alpha_{g,\omega}(x)$ as $\alpha(x)$. We take $\nabla_\Gamma$, the vector fields 
spanning the tangent bundle of $\Gamma$,  to be
\[
(\nabla_\Gamma f)(x,\alpha(x)) := 
\nabla_g ( f(x, \alpha(x) ) ) = g^{-1} \nabla (f (x,\alpha(x)) ) = g^{-1} ( \nabla f + \nabla \alpha \, \pa_t f )
= \nabla_g f + \nabla_g \alpha \, \pa_t f,
\]
and note that
\begin{align}
\| \nabla_g ( f (x, \alpha(x) ) ) \|^2 &= \| \nabla_g f \|^2 + (\pa_t f)^2 + 2 (\pa_t f) \, \la \nabla_g \alpha , \nabla_g f \ra
\nn
\\
& = (\nabla f)^T g^{-1} (\nabla f) + (\pa_t f )^2 + 2 (\pa_t f) (\nabla \alpha)^T g^{-1} \nabla f
\nn
\\
& = a_{ij} \, \pa_i f \, \pa_j f + 2 (\pa_t f) \, a_{ij} \pa_i \alpha \, \pa_j f + (\pa_t f)^2.
\label{eq:fgnorm}
\end{align}

From Proposition \ref{prop:exinj}, 
we know that the part of the surface $t=\alpha(x)$ in $\Bbar \times \R$, is traced out by the 
family of curves $t \to (\gamma_(t), t)$, the level surfaces of $\alpha$ split the region $\Bbar$ into at most 
two connected components, and  $\min_{x \in \Bbar} \alpha(x)=-1$. For each $\tau \in [-1, T]$, define (see Figure \ref{fig:energy1})
\begin{figure}[h]
\begin{center}
\epsfig{file=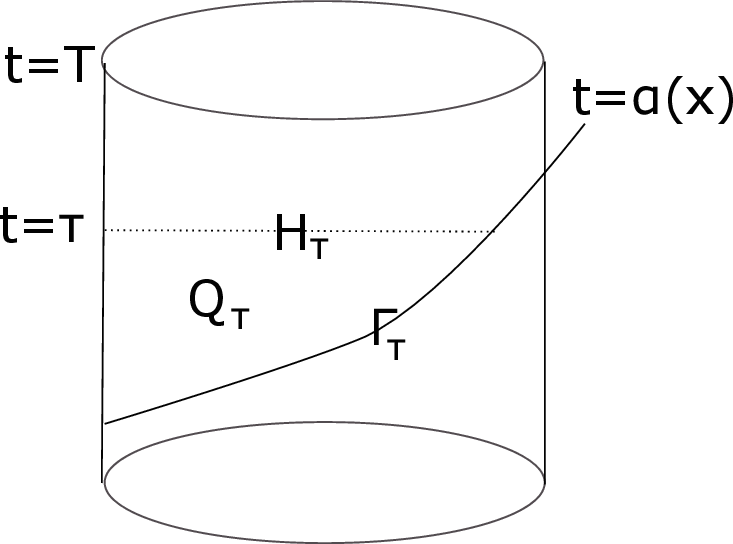, height=1.5in}
\end{center}
\caption{The $\tau$ dependent geometrical objects.}
\label{fig:energy1}
\end{figure}
\begin{gather*}
D_{\tau} := \{ x \in \Bbar : \alpha(x) \leq \tau \}, \qquad Q_\tau := \{ (x,t) : x \in \Bbar , ~ \alpha(x) \leq t \leq \tau  \}, 
\\
H_{\tau} := \{ (x, \tau) : x \in D_{\tau} \},
\qquad
\Gamma_\tau := \{ (x, \alpha(x)) : x \in D_{\tau} \}.
\end{gather*}
Note that if $\tau \geq \max_{x \in \Bbar} \alpha(x)$ then
$H_\tau = \Bbar \times \{\tau\}$ and $\Gamma_\tau = \Gamma$.
Define the `energy' at level $\tau$ as
\begin{align*}
E(\tau, \sigma) & := \int_{H_\tau}   |\pa_t f|^2 + a_{ij} \pa_i f \, \pa_j f  + \sigma^2 |f|^2.
\end{align*}
In our calculations below, there will be integrals over $\Sigma_+$ of expressions which are bilinear in
$\nabla_{x,t} f, f$ and there will be no $\sigma$ in these expressions. All such terms will be denoted by
$I_{\Sigma_+}(f)$.

One may verify the identity
\begin{align}
2 (\L f + \sigma^2 f) \, \pa_t f & = \pa_t \left ( a_{ij} \, \pa_i f \, \pa_j f + (\pa_t f)^2 + \sigma^2 f^2 \right )
- 2 \pa_j \left (a_{ij} \, \pa_i f \, \pa_t f \right )
+ 2(P f) \, f_t,
\label{eq:Afiden}
\end{align}
where $P f =  \pa_j (a_{ij}) \pa_i f -  b_i \pa_i f - cf$. 
Integrating \eqref{eq:Afiden} over $Q_\tau$, applying the divergence theorem and using \eqref{eq:fgnorm}, we have
\begin{align}
2\int_{Q_\tau} &  (\L f + \sigma^2 f - Pf ) \, f_t    +  I_{\Sigma_+}(f)
= \int_{H_\tau} (\pa_t f)^2 + a_{ij} \pa_i f \pa_j f + \sigma^2 |f|^2 
\nn
\\
& \qquad  -
\int_{D_{\tau}} \left ( a_{ij}  \, \pa_i f \, \pa_j f + 2 \pa_t f \, a_{ij} \; \pa_i \alpha \; \pa_j f + (\pa_t f)^2
+ \sigma^2 f^2\right) (x, \alpha(x)) 
\nn
\\
& = E(\tau, \sigma) - \int_{D_\tau} ( |\nabla_\Gamma f|^2 + \sigma^2 f^2)(x,\alpha(x)) 
\label{eq:divthm}
\end{align}
where
\begin{equation}
I_{\Sigma_+} (f) = 2 \int_{\Sigma_+} a_{ij} \nu_j \pa_i f \, \pa_t f .
\label{eq:Isigma}
\end{equation}

We first prove \eqref{eq:fHTest}. From \eqref{eq:divthm}. Taking $\sigma=1$, we have
\begin{align*}
E(\tau, \sigma=1) + I_{\Sigma_+}(f) & \cleq \int_\Gamma |\nabla_\Gamma f|^2 +  |f|^2 
+ \int_{Q_\tau} (|\L f| +  |f| + |\nabla_{x,t} f| )\, |f_t|
\\
& \cleq \int_\Gamma |\nabla_\Gamma f|^2 +  |f|^2  + \int_{Q_\tau} |\L f|^2 + |\nabla_{x,t}f|^2 +  |f|^2
\\
& \cleq \int_{-1}^\tau E(s, \sigma=1) \, ds + \int_\Gamma |\nabla_\Gamma f|^2 +  |f|^2  + \int_{Q_\tau} |\L f|^2.
\end{align*}
Hence, by Gronwall's inequality, 
\[
E(T,\sigma=1) \cleq  \int_\Gamma |\nabla_\Gamma f|^2 +  |f|^2  
+ \int_{Q_+} |\L f|^2 \, dV + \int_{\Sigma_+} |\nabla_{x,t}f|^2 + f^2 ,
\]
with the constant independent of $f$,
proving \eqref{eq:fHTest}.

Next we establish \eqref{eq:fGammaest}. We take $\tau \in [T_0, T]$ and note that $H_\tau = B \times \{\tau\}$,
$\Gamma_\tau = \Gamma$. From \eqref{eq:divthm}, for $\sigma \geq 1$,
we have
\begin{align*}
\int_\Gamma | \nabla_\Gamma f|^2 & + \sigma^2 |f|^2  + I_{\Sigma_+}(f) 
 \cleq \int_{H_\tau} |\nabla_{x,t}f|^2 + \sigma^2 f^2 
+  \int_{Q_+} (|\L f| + |\nabla_{x,t}f| + \sigma^2 |f|) \, |f_t| 
\\
& \cleq \int_{H_\tau} (|\nabla_{x,t}f|^2 + \sigma^2 f^2)  +  \int_{Q_+} |\L f|  \, |f_t|  
+ \sigma \int_{Q_+}  |\nabla_{x,t}f|^2 + \sigma^2 |f|^2 .
\end{align*}
Integrating this in $\tau$ over $[T_0,T]$, for $\sigma \geq 1$, we obtain
\begin{align*}
\int_\Gamma | \nabla_\Gamma f|^2 + \sigma^2 |f|^2  
\cleq  \int_{Q_+} |\L f|  \, |f_t| + \sigma \int_{Q_+}  |\nabla_{x,t}f|^2 + \sigma^2 |f|^2 
+ \int_{\Sigma_+} |\nabla_{x,t} f|^2 + |f|^2,
\end{align*}
thus proving \eqref{eq:fGammaest}.
\end{proof}

%%%%%%%%%%%%%%
%%%%%

\subsection{Proof of Proposition \ref{prop:phiexistence}}\label{subsec:weight}

Since $\kappa$ is strictly convex in $\Bbar$, there is a positive constant  $\kappa_{con}$ such that
\begin{align}
\nabla^2 \kappa(x) \geq \kappa_{con} \, g(x), \qquad x \in \Bbar,
\label{eq:kappamin}
\end{align}
as quadratic forms on $T_x(\R^n)$.
Define
\[
T_2 := 2 \sqrt{g_{max}} -1,
\]
and note that
\[
T_2 \geq \max_{x \in \Bbar} \alpha(x), \qquad T_0 \geq 1.
\]
We will seek a $T_*$ which is larger than $2 T_2$, so consider any $T$ with 
\[
T \geq \max\{ T_0, 2T_2\};
\] 
hence $T \geq 2$.
For any small $\ep>0$, define
\[
\psi(x,t) := \|\kappa\|_\infty + \kappa(x) - \ep \, (t - \alpha(x))^2, \qquad (x,t) \in \Bbar \times \R.
\]
We added the $\|\kappa\|_\infty$ term in the definition of $\psi$ to help with (c); this has no effect on the proof of
(a) and (b). One could also have replaced the convex $\kappa$ by the convex $\kappa + \|\kappa\|_\infty$.

Now
\[
\min_{x \in \Bbar} \psi(x, \alpha(x))  - \|\kappa\|_\infty =  \min_{x \in \Bbar}  \kappa(x).
\]
Further, since $\max_{x \in \Bbar} \alpha(x) \leq T_2$ and $ \min_{x \in \Bbar} \alpha(x) = -1$, we have
\begin{align*}
\max_{ x \in \Bbar} |\pm T-\alpha(x)| &\leq \max(T+1, T_2 + T) = T + T_2 \leq \frac{3T}{2},
\\
\min_{x \in \Bbar} |\pm T - \alpha(x)| & \geq \min ( T-1,  T -T_2 ) = T-T_2 \geq \frac{T}{2},
\end{align*}
hence
\begin{align*}
\max_{x \in \Bbar} \psi(x, \pm T) - \| \kappa \|_\infty
& \leq \max_{x \in \Bbar} \kappa(x) - \ep \min_{x \in \Bbar} |\pm T - \alpha(x)|^2
\leq \max_{x \in \Bbar} \kappa(x) - \ep \frac{T^2}{4}.
\end{align*}
Therefore
\begin{align*}
2 \delta := \min_{x \in \Bbar} \psi(x,\alpha(x)) - \max_{x \in \Bbar} \psi(x,\pm T) 
& \geq \ep \frac{T^2}{4} - ( \max_{x \in \Bbar} \kappa(x) - \min_{x \in \Bbar} \kappa(x)),
\end{align*}
with $\delta$ positive if
\begin{equation}
\ep T^2   >  4 \left ( \max_{x \in \Bbar} \kappa(x) - \min_{x \in \Bbar} \kappa(x) \right ).
\label{eq:condTk1}
\end{equation}

Next, if $ t \to (x(t),t)$ is a time parametrized ray for $\Box_g$ then 
then $t \to x(t)$ is a unit speed geodesic for $(R^n,g)$. For $(x(t),t) \in \Bbar \times [-T,T]$, 
using the definition of the covariant derivative and that $t \to x(t)$ is a geodesic, we have 
\begin{align*}
\frac{d^2}{dt^2} [ \psi(x(t),t ) ]
& = \frac{d^2}{dt^2}[ \kappa( x(t) ) ] - 2 \ep \left ( 1 - \frac{d}{dt} [ \alpha(x(t) ) ]\right )^2
+ 2 \ep (t - \alpha(x(t))) \frac{d^2}{dt^2} [ \alpha(x(t) ) ]
\\
& = (\nabla^2\kappa) ( \xdot, \xdot) - 2 \ep( 1 - \la \nabla_g \alpha, \xdot \ra )^2 + 2 \ep (t-\alpha) (\nabla^2 \alpha)(\xdot, \xdot).
\end{align*}
Now 
\[
| \la \nabla_g \alpha, \xdot \ra | \leq \|\nabla_g \alpha \|^2 \, \|\xdot\|^2 = 1,
\]
and
\[
\max_{x \in \Bbar, ~|t| \leq T} |t-\alpha(x)| \leq \max_{x \in \Bbar} T + \max_{x \in \Bbar} |\alpha(x)|
\leq T + T_2 \leq \frac{3T}{2}.
\]
Further, there is a constant $c_\alpha$, dependent on $\|\alpha\|_{C^2}$ and $g$ such that
\[
\nabla^2 \alpha \leq c_\alpha g \qquad \text{on } \Bbar.
\]
Hence, for $(x(t),t) \in \Bbar \times [-T,T]$, we have
\begin{align*}
\frac{d^2}{dt^2} [ \psi(x(t),t ) ] & \geq (\nabla^2 \kappa - 3 T \ep \, \nabla^2 \alpha)(\xdot, \xdot) - 8 \ep
\geq (\kappa_{con} - 3 T \ep c_\alpha) g(\xdot, \xdot) - 8 \ep
\\
& = \kappa_{con} - \ep (8 +3 T \,c_\alpha) \geq \kappa_{con} - \ep T (4 + 3 c_\alpha),
\end{align*}
since we chose $T  \geq 2$.
So level surfaces of $\psi$ are pseudoconvex w.r.t $\Box_g$ on $\Bbar \times [-T,T]$, if 
\begin{equation}
 \ep T (4+3 \,c_\alpha)  <  \kappa_{con}.
\label{eq:pseudocond}
\end{equation}

Now \eqref{eq:condTk1} and \eqref{eq:pseudocond} hold iff
\begin{equation}
\frac{4}{T^2} ( \max_{x \in \Bbar} \kappa(x) - \min_{x \in \Bbar} \kappa(x))
< \ep < \frac{ \kappa_{con}}{(4+ 3 c_\alpha) T}.
\label{eq:epdef}
\end{equation}
So there is an $\ep>0$ so that
\eqref{eq:condTk1} and \eqref{eq:pseudocond} hold if $T \geq 2 T_2$ and
\[
\frac{4}{T} ( \max_{x \in \Bbar} \kappa(x) - \min_{x \in \Bbar} \kappa(x)) <  \frac{ \kappa_{con}}{(4+ 3 c_\alpha) }.
\]
So it is enough choose $T \geq T_*$ where
\[
T_* := 2T_2 +  \frac{4(4 + 3 c_\alpha)}{\kappa_{con}} \, 
( \max_{x \in \Bbar} \kappa(x) - \min_{x \in \Bbar} \kappa(x)),
\]
and then choose an $\ep_T$ satisfying \eqref{eq:epdef}.

If the level surfaces of $\psi$ are pseudoconvex with respect to $\Box_g$ on the region $\Bbar \times [-T,T]$ then
\cite[Propositions A.3, A.4, A.5]{rs20a} guarantees there is a $\lambda_*>0$ so that if $\lambda \geq \lambda_*$
then
\[
\vphi(x,t) = e^{\lambda \psi(x,t)} = e^{\lambda \left ( \|\kappa\|_\infty + \kappa(x) - \ep (t - \alpha(x))^2 \right )}
\]
is strongly pseudoconvex with respect to $\Box_g$ on the region $\Bbar \times [-T,T]$. Note that $\lambda_*$ depends
on $T, \ep, \|\kappa\|_{C^2(\Bbar)}, \|\alpha\|_{C^2(\Bbar)}$, $\kappa_{con}$. Further
(b) holds for $\vphi$ because it holds for $\psi$.

We now verify that $\vphi$ satisfies (c). We use a modification of the proof of \cite[Lemma 3.2]{rs20b} 
- we demand more of our $\eta(\sigma)$ than was needed in \cite{rs20b}.
We first observe that
\[
\psi(x,\alpha(x)) \geq \| \kappa \|_\infty + \min_{x \in \Bbar} \kappa(x) \geq 0, \qquad \forall x \in \Bbar.
\]
Next, for any $x \in \Bbar$, $t \in \R$, $\lambda \geq \lambda_*$, 
\begin{align*}
 \vphi(x,\alpha(x)) - \vphi(x,t)& =  e^{\lambda \psi(x,\alpha(x))} - e^{\lambda \psi(x,t)} 
= 
e^{\lambda \psi(x,\alpha(x))} \left ( 1 - e^{ \lambda ( \psi(x,t) - \psi(x,\alpha(x)) )} \right )
\\
& = e^{\lambda \psi(x,\alpha(x))}  \left ( 1 - e^{ -\ep \lambda ( t - \alpha(x) )^2} \right )
\geq 1 - e^{ -\ep \lambda ( t - \alpha(x) )^2},
\end{align*}
because $\psi(x, \alpha(x)) \geq 0$ and $\lambda \geq 0$.
Now, for $s \geq 0$,
\[
1 - e^{-s} \geq \frac{1}{3} \min(1,s),
\]
hence, for $(x,t) \in \Bbar \times [-T,T]$,
\[
3 (\vphi(x,t) - \vphi(x,\alpha(x)) \leq - \min \left (1, \, \ep \lambda (t - \alpha(x))^2 \right ),
\]
therefore, for each $x \in \Bbar$, we have
\begin{align*}
\int_{-T}^T e^{ 2 \sigma (\vphi(x,t) - \vphi(x,\alpha(x))} \, dt
& \leq 
\int_{-T}^T e^{-(2\sigma /3)\, \min( 1, \, \ep \lambda(t - \alpha(x))^2)} \, dt
= 
\int_{-T - \alpha(x)}^{T- \alpha(x)} e^{- (2 \sigma /3) \, \min( 1, \, \ep \lambda t^2)} \, dt
\\
& \leq \int_{-T - C_1}^{T+C_2} e^{- (2\sigma /3) \, \min( 1, \, \ep \lambda t^2)} \, dt,
\end{align*}
for some $C_1, C_2$ independent of  $\sigma$ and $x \in \Bbar$. Hence
\[
\eta(\sigma) := \sup_{x \in \Bbar} \int_{-T}^T e^{ 2 \sigma (\vphi(x,t) - \vphi(x,\alpha(x))} \, dt
\leq 
\int_{-T - C_1}^{T+C_2} e^{- (2 \sigma/3) \, \min( 1, \, \ep \lambda t^2)} \, dt.
\]
Noting that $\ep>0, \lambda>0$, by the Dominated Convergence Theorem,
\[
\lim_{\sigma \to \infty} \sigma \eta(\sigma) =0.
\]
%

%%%%%%%%%%%%%
%%%%%%%%%%%%%%%

\subsection{Consequence of a Carleman estimate}\label{subsec:conseq}

The endings of the proofs of Theorems \ref{thm:cunique}, \ref{thm:Aunique} reduce
to an application of the following proposition. Below
\[
\|[u,v,w]\|_{0,\sigma,Q}^2 := \|u\|_{0,\sigma,Q}^2 + \|v\|_{0,\sigma,Q}^2 + \|w\|_{0,\sigma,Q}^2.
\]

\begin{prop}\label{prop:conseq}
Suppose $T>0$, $E$ a second order elliptic operator on $\Bbar$, $\alpha(x)$ a smooth function on $\Bbar$,
and $\vphi(x,t)$ a smooth function on $Q$, such that the following conditions hold:
\vspc
\begin{enumerate}[(a)]
\item $\mu(x) := \vphi(x,\alpha(x))$ is strongly pseudoconvex w.r.t $E$ on $\Bbar$;
\item $ \vphi(x, \alpha(x)) > \vphi_H := \sup_{H_\pm} \vphi$ for all $x \in \Bbar$;
\item $\lim_{\sigma \to \infty} \sigma \eta(\sigma)=0$ where
\[
\eta(\sigma):= \sup_{x \in \Bbar} \int_{-T}^T e^{2\sigma( \vphi(x,t) - \vphi(x,\alpha(x)))} \, dt.
\]
\end{enumerate}
\vspc
If $f \in C_c^\infty(B)$ and $F(\sigma)$ is a non-negative function on $[0,\infty)$ such that
\begin{equation}
\|Ef\|_{0,\sigma,\Gamma}^2 - C \|f\|_{1,\sigma,\Gamma}^2 
\cleq  \| [\pa^2 f, \pa f, f] \|_{0,\sigma,Q}^2 +  e^{2 \sigma \vphi_H} \|[\pa^2 f, \pa f, f ]\|_{0,Q}^2 + F(\sigma),
\label{eq:lemmaEF}
\end{equation}
for some $C>0$ and all large enough $\sigma$, with $C$ and the inequality
constant independent of $\sigma$, then
\[
 \sigma^{-1} \|\pa^2 f \|_{0,\sigma,\Gamma} 
 + \sigma \| f \|_{1,\sigma,\Gamma}^2 \cleq F(\sigma),
 \]
 for large enough $\sigma$, with the inequality constant independent of $\sigma$.
\end{prop}

\begin{proof}~

Define $\delta$ by
\[
2 \delta := \inf_{x \in \Bbar} \vphi(x,\alpha(x)) - \vphi_H;
\]
then $\delta>0$ by hypothesis.

Since $\mu$ is strongly pseudoconvex w.r.t $E$ on $\Bbar$, and $f$ and its derivatives are zero on $\pa B$,
from \cite[Theorem 8.3.1]{hormander1976}, we have
\[
\sigma^{-1} \int_B e^{2 \sigma \mu} |\pa^2 f|^2
+ \sigma \int_B e^{2 \sigma \mu} (|\pa f|^2 + \sigma^2 |f|^2)
\cleq 
\int_B e^{2 \sigma \mu} |E f|^2,
\]
for large enough $\sigma$, with the constant independent of $\sigma$.
On the other hand, for any smooth function $h(x)$ on $\Bbar$, we have
\begin{align*}
\|h\|_{0,\sigma,Q}^2 = \int_Q e^{2 \sigma \vphi} |h|^2
= \int_B e^{2 \sigma \vphi(x,\alpha(x))} |h(x)|^2 \int_{-T}^T e^{2 \sigma (\vphi(x,t) - \vphi(x,\alpha(x)))} \, dt \, dx
\leq \eta(\sigma) \|h\|_{0,\sigma,\Gamma}^2.
\end{align*}
and
\begin{align*}
e^{2 \sigma \vphi_H} \|h\|_{0,Q}^2 
= e^{2 \sigma \vphi_H} \int_Q |h|^2
\leq  e^{-2 \sigma \delta} \int_B e^{2 \sigma \vphi(x,\alpha(x))} |h(x)|^2
= e^{-2 \sigma \delta} \|h\|_{0,\sigma,\Gamma}^2.
\end{align*}
Using these in \eqref{eq:lemmaEF}, we obtain
\[
\sigma^{-1} \|\pa^2 f\|_{0,\sigma,\Gamma}^2
+ \sigma \|f\|_{1,\sigma, \Gamma}^2
\cleq  ( \eta(\sigma) + e^{-2 \sigma \delta} )  \|[\pa^2 f, \pa f, f]\|_{0,\sigma,\Gamma}^2 + F(\sigma),
\]
for large enough $\sigma$. Now $\lim_{\sigma \to \infty} \sigma \eta(\sigma) =0$ and
$\lim_{\sigma \to \infty} \sigma e^{-2 \sigma \delta} =0$. Hence, for large enough $\sigma$, we have
\[
\sigma^{-1} \|\pa^2 f\|_{0,\sigma,\Gamma}^2
+ \sigma \|f\|_{1,\sigma, \Gamma}^2
\cleq  F(\sigma).
\]
\end{proof}

%%%%%%%%%%%%%
%%%%%%%%%%%%%

\subsection{Some useful calculations}\label{subsec:standard}

\noindent
\underline{Derivation of \eqref{eq:Lwpta}}

Suppose  $ \L := \pa_t^2 - a_{ij} \pa_i \pa_j - b_i \pa_i - c$, $\L'$ is an arbitrary operator,
$f'(x,t)$ is a smooth function on $\Bbar \times \R$, and define
\[
\fhat'(x,t) := f'(x, t - \alphabar(x)).
\]
Then
\begin{align*}
\pa_t^2 \fhat'(x,t) & = (\pa_t^2 f')(x, t - \alphabar)
\\
\pa_j \fhat'(x,t) & =(\pa_j f')(x, t - \alphabar(x)) - (\pa_j \alphabar) (\pa_t f')(x, t - \alphabar(x))
\\
\pa_i \pa_j \fhat'(x,t) & = (\pa_i \pa_j f')(x, t - \alphabar(x)) - (\pa_i \alphabar) (\pa_t \pa_j f')(x, t - \alphabar(x))
  - (\pa_j \alphabar) (\pa_t \pa_i f')(x, t - \alphabar(x))
  \\
& \qquad -  (\pa_i \pa_j \alphabar) (\pa_t f')(x, t - \alphabar(x))
+ (\pa_i \alphabar) \, (\pa_j \alphabar) (\pa_t^2 f')(x, t - \alphabar(x)).
\end{align*}
Hence
\begin{align*}
(\L \fhat')(x,t) & = F(x,t-\alphabar(x)),
\end{align*}
where
\begin{align*}
F  & = \L f' + 2 a_{ij} \pa_i \alphabar \, \pa_t \pa_j f' - a_{ij} \pa_i \alphabar \pa_j \alphabar \pa_t^2 f' 
+ (\calE - c) \alphabar \, \pa_t f'
\\
& = (\L - \L') f' + \L' f' + 2 (\nabla \alphabar)^T A \, \nabla (\pa_t f') - (\nabla \alphabar)^T A (\nabla \alphabar)\, \pa_t^2 f' 
+ (\calE - c) \alphabar \, \pa_t f'.
\end{align*}

%%%%%%%%

\noindent
\underline{Derivation of \eqref{eq:firstalphabar}}

Here 
\[
\calE := a_{ij} \pa_i \pa_j + b_i \pa_i + c, \qquad
\calE' := a'_{ij} \pa_i \pa_j + b_i \pa_i + c,
\qquad \fbar_0 = f_0 - f'_0.
\]
Since
\begin{align*}
2 (\nabla \alpha)^T A (\nabla f_0) + ((\calE -c) \alpha) f_0 & =0, \qquad \text{on } \Bbar,
\\
2 (\nabla \alpha')^T A' (\nabla f'_0) + ((\calE' -c) \alpha') f_0' &=0, \qquad \text{on } \Bbar.
\end{align*}
we have
\begin{align*}
\left ( 2 (\nabla \alpha)^T A \nabla + (\calE-c) \alpha \right ) \fbar_0
& = - \left ( 2 (\nabla \alpha)^T A \nabla +  (\calE-c) \alpha \right ) f'_0
\\
& = - \left ( 2 (\nabla \alphabar)^T A \nabla + (\calE-c) \alphabar \right ) f'_0
- \left ( 2 (\nabla \alpha')^T A\nabla + (\calE -c) \alpha' \right ) f'_0
\\
& = - \left ( 2 (\nabla \alphabar)^T A \nabla + (\calE-c) \alphabar) \right ) f'_0
- \left ( 2 (\nabla \alpha')^T \Abar \nabla + (\calE-\calE') \alpha' \right ) f'_0
\\
& \qquad -  \left ( 2 (\nabla \alpha')^T A' \nabla +  (\calE' -c)\alpha' \right ) f'_0
\\
& = - \left ( 2 (\nabla \alphabar)^T A \nabla + (\calE -c) \alphabar \right ) f'_0
- \left ( 2 (\nabla \alpha')^T \Abar \nabla + \abar_{ij} \pa_i \pa_j \alpha' \right ) f'_0.
\end{align*}

%%%%%%%

\noindent
\underline{Derivation of \eqref{eq:LUmalpha}}

For any function $f(x)$ on $D$, we have defined (using Einstein's summation convention) the operators
\begin{align*}
\calE f = a_{ij} \pa_i \pa_j f + b_i \pa_i f + cf, 
\qquad \T f  =  a_{ij} \pa_i \alpha \, \pa_j f.
\end{align*}
Fix a positive integer $m$ and define
\[
U_m(x,t) = \sum_{k=-1}^{m} f_k(x) K_k (t- \alpha(x)), \qquad (x,t) \in D \times \R.
\]
For $k \geq -1$, on $D \times \R$
\begin{align*}
\pa_j \left ( f_k(x) K_k(t- \alpha(x)) \right )
& = \pa_j f_k \, K_k(t- \alpha(x)) - f_k \, \pa_j \alpha \, K_{k-1}(t-\alpha(x))
\\
\pa_i \pa_j \left ( f(x) K_k(t- \alpha(x)) \right )
& = \pa_i \pa_j f_k \, K_k(t-\alpha(x)) - \pa_j f_k \, \pa_i \alpha \, K_{k-1}(t-\alpha(x))
- \pa_i f_k \, \pa_j \alpha \,  K_{k-1}(t-\alpha(x))
\\
& \qquad - f_k \pa_i \pa_j \alpha K_{k-1}(t-\alpha(x))
+ f_k \, \pa_i \alpha \, \pa_j \alpha \, K_{k-2}(t - \alpha(x)).
\end{align*}
Hence, using the Einstein summation convention and that $a_{ij} \pa_i \alpha \, \pa_j \alpha =1$, we have
\begin{align*}
a_{ij} \, \pa_i \pa_j \left ( f(x) K_k(t- \alpha(x)) \right )
& = (a_{ij} \pa_i \pa_j f_k) K_k(t-\alpha)  - \left ( 2 a_{ij} \pa_i \alpha \, \pa_j f_k  + a_{ij} \pa_i \pa_j \alpha \, f_k \right )
\, K_{k-1}(t-\alpha) 
\\
& \qquad + f_k K_{k-2}(t - \alpha),
\\
b_i \, \pa_i  \left ( f_k(x) K_k(t- \alpha(x)) \right ) 
& = b_i \pa_i f_k \, K_k(t-\alpha) - ( b_i \pa_i \alpha) f_k K_{k-1}(t-\alpha).
\end{align*}
We also have
\[
\pa_t^2 \left ( f_k(x) K_k(t- \alpha(x)) \right )  = f_k \, K_{k-2}(t - \alpha(x)).
\]
Therefore
\begin{align*}
\L U_m & = \sum_{k=0}^{m} \left (  2 a_{ij} \pa_i \alpha \, \pa_j f_k  + a_{ij} \pa_i \pa_j \alpha \, f_k 
+ b_i \pa_i \alpha f_k - a_{ij} \pa_i \pa_j f_{k-1} - b_i \pa_i f_{k-1} - c f_{k-1} \right ) K_{k-1}(t-\alpha(x))
\\
& \qquad + ( 2 a_{ij} \pa_i \alpha \, \pa_j f_{-1}  + a_{ij} \pa_i \pa_j \alpha \, f_{-1}
+ b_i \pa_i \alpha f_{-1} ) K_{-2}(t-\alpha(x) ) 
\\
& \qquad + ( - a_{ij} \pa_i \pa_j f_m - b_i \pa_i f_m - c f_m) K_m(t - \alpha(x))
\\
& = \sum_{k=0}^m ( (2\T + (\calE - c) \alpha) f_k - \calE f_{k-1})(x) \,  K_{k-1}(t-\alpha(x)) 
+ (2\T + (\calE - c) \alpha) f_{-1}(x) \, K_{-2}(t-\alpha(x))
\\
& \qquad - (\calE f_m)(x) \, K_m(t-\alpha(x)).
\end{align*}

%%%%%%%

%%%%%%%%%%%%%%%%%%%%
%%%%%%%%%%%%%%%%%%%%%
\bibliographystyle{plain}
\bibliography{references}

%%%%%%%%%%%%%%%
\end{document}